\documentclass[pdflatex,sn-mathphys-num]{sn-jnl}

\usepackage{graphicx}
\usepackage{multirow}
\usepackage{amsmath,amssymb,amsfonts}
\usepackage{amsthm}
\usepackage{mathrsfs}
\usepackage[title]{appendix}
\usepackage{xcolor}
\usepackage{textcomp}
\usepackage{manyfoot}
\usepackage{booktabs}
\usepackage{listings}

\newtheorem{theorem}{Theorem}[section]
\newtheorem{definition}{Definition}[section]
\newtheorem{lemma}[theorem]{Lemma}
\newtheorem{corollary}[theorem]{Corollary}
\newtheorem{proposition}[theorem]{Proposition}
\newtheorem{remark}[theorem]{Remark}

\begin{document}

\title[Foliation by bi-rotational self-shrinkers in Euclidean space]{Foliation by bi-rotational self-shrinkers in Euclidean space}

\author*[1]{\fnm{Junyoung} \sur{Park}}\email{jp2453@math.rutgers.edu}

\abstract{In this note, we first construct a family of bi-rotational asymptotically conical self-shrinkers with boundary called `trumpets', and a  family of bi-rotational self-shrinkers with boundary which close up called `disks'. We then show that these self-shrinkers together with two generalized cylinders, and the quadratic minimal cone foliate the whole space except some compact set containing the origin. }


\maketitle
\tableofcontents
\section{Introduction}
A smooth one-parameter family of hypersurfaces $(M_t^n)_{t  \in I}$ in Euclidean space $\mathbf{R}^{n+1}$ is a solution to mean curvature flow (henceforth MCF) if its normal velocity is the mean curvature, i.e
\begin{equation}
    \langle \partial_tX_t, \nu_t \rangle  = -H_t.
\end{equation}
Here, $\nu_t$ is the unit normal of $M_t$ and $H = \textit{{\normalfont div}}_{M_t}\nu_t$ is the mean curvature of $M_t$. Since the seminal work of Brakke \cite{Br}, mean curvature flow has been extensively studied for the past several decades. \\

Since MCF is invariant under various Euclidean motions such as dilation and translation, one is naturally led to investigate special solutions to the flow which moves only by such motions. Among these self-similar solutions, there are solutions to MCF which simply shrinks down forward in time. These are called self-shrinking solutions. \\

Self-shrinking solutions are particularly important in MCF since they model finite time singularities in MCF. More precisely, if $\mathcal{M} = (M_t)_{t \in [0, T)}$ is a solution to MCF which forms a finite time singularity at $X_0 = (x_0, T)$, one can consider the (possibly weak) limit of the following blowup sequence
\begin{equation}
    M^i_t = \lambda_i(M_{T + \lambda^{-2}_it} - x_0), \ \lambda_i \to \infty.
\end{equation}
Such a limit flow is called the tangent flow of $\mathcal{M}$ at $X_0$. It is well known by the Huisken's monotonicity formula \cite{Hu} that tangent flows must be a self-shrinking flow given by
\begin{equation}
    M^{\infty}_t = \sqrt{-t}M^{\infty}_{-1}, \ t < 0.
\end{equation}
This connection to singularity analysis is one fundamental reason why self-shrinking solutions have received much attention in MCF theory. \\

In order to study self-shrinking solutions, it is enough to study the $t = -1$ time slice. Then the MCF equation reduces to the following elliptic type equation given by
\begin{equation}\label{self shrinker equation}
    \Tilde{H} = \frac{1}{2}\langle \Tilde{X}, \nu \rangle.
\end{equation}
Self-shrinkers are precisely solutions to the above equation.\\

There has been a lot of work on the construction of self-shrinkers. One example is the work of Kleene-M\o ller \cite{KM}, in which they constructed a family of rotationally symmetric self-shrinkers with boundary. Later, Angenent-Daskalopolous-\v{S}e\v{s}um \cite{ADS} used the previously mentioned self-shrinkers together with one more family of rotationally symmetric self-shrinkers they constructed to obtain a key estimate that was extensively used to study precise asymptotics of ancient ovals asymptotic to the round cylinder \\
$\mathbf{S}^{n-1}(\sqrt{2(n-1)}) \times \mathbf{R}$. This key estimate, which the authors refer to as inner-outer estimate, is a consequence of the fact that the self-shrinkers together with the round cylinder foliate a conical neighborhood of the cylinder. \\

It is natural to ask if one can obtain analogous foliations by self-shrinkers near generalized cylinders $\mathbf{S}^{n-k}(\sqrt{2(n-k)}) \times \mathbf{R}^k$, since such a foliation can be a crucial tool in the study of more general ancient ovals. It turns out that one can use `subsolutions' to the self-shrinker equations obtained by `rotating' the previously used self-shrinkers, and still carry out analogous analysis in the more general setting (e.g. Haslhofer-Du \cite{Du}). However, it would indeed be more satisfying if one can find actual self-shrinkers in the more general symmetry setting. \\

The goal of this paper is to show that this can be done.
\begin{theorem}\label{Main theorem : Existence of birotational trumpets}
For each $2 \leq k \leq n-1$, there exists a one-parameter family of `trumpets' which are $O(k) \times O(n-k+1)$ symmetric, asymptotically conical self-shrinkers diffeomorphic to $[0, \infty) \times \mathbf{S}^{k-1} \times \mathbf{S}^{n-k}$.     
\end{theorem}
\begin{theorem}\label{Main theorem : existence of birotational caps}
   For each $2 \leq k \leq n-1$, there exists a one-parameter family of `disks' which are $O(k) \times O(n-k+1)$ symmetric self-shrinkers diffeomorphic to $\mathbf{S}^{k-1} \times \mathbf{D}^{n-k+1}$.
\end{theorem}
\begin{theorem}\label{Main theorem : Foliation}[See theorem \ref{specific form of theorem 1.3}]
    For $2 \leq k \leq n-1$, there exists a compact subset $0 \in K(k,n) \subset \mathbf{R}^{n+1}$ so that $\mathbf{R}^{n+1} - K$ is foliated by `trumpets', `disks', generalized cylinders and the $O(k) \times O(n-k+1)$ `minimal cone'.
\end{theorem}
\begin{remark}
    The foliation in theorem \ref{Main theorem : Foliation} essentially satisfies all the key estimates of the foliation given in \cite{ADS} and \cite{Du}. In particular, our foliation can alternatively be used in the analysis of ancient ovals with bi-rotationally symmetry, instead of the foliation by `sub' self-shrinkers given by \cite{Du}.
\end{remark}

The outline of the paper is as follows. In section \ref{prereq}, we provide several preliminary calculations related to bi-rotationally symmetric self-shrinkers. In section \ref{section 3}, we prove theorem \ref{Main theorem : Existence of birotational trumpets}, and analyze the behavior of `trumpets'. In section \ref{section 4}, we prove theorem \ref{Main theorem : existence of birotational caps}, and analyze the behavior of `disks'. Finally in section \ref{section 5}, we prove theorem \ref{Main theorem : Foliation}. 
\section{Preliminary calculations}\label{prereq}
In this section, we collect some preliminary materials, and define frequently used constants. The details of the computations in this section can be found in \cite{survey}. \\

For $2 \leq k \leq n-1$, any $O(k) \times O(n-k+1)$ symmetric hypersurface can be identified by a curve in the first quadrant $I \ni s \to (x(s), y(s)) \in \mathbf{R}_+^2$ via
\begin{equation}
     I \times \mathbf{S}^{k-1}  \times \mathbf{S}^{n-k}\ni (s, w_1, w_2) \to(x(s)w_1, y(s)w_2) \in \mathbf{R}^{n+1}.
\end{equation}
We will consider the case when $y$ (resp. $x$) is given by a function of $x$ (resp. $y$). When  $y = u(x)$, the unit normal, and the principal curvatures are given by
\begin{equation}\label{geometric quantities calculation}
    \begin{cases}
    &\nu = \frac{1}{\sqrt{1 + (u')^2}}(-u'(x)w_1, w_2) \\ & \lambda_1 = -\frac{u''}{(1 + (u')^2)^{3/2}}  \\ &\lambda_i = -\frac{u'}{x(1 + (u')^2)^{1/2}} \text{ for }i = 2,3,..,k  \\ & \lambda_i = \frac{1}{u(1 + (u')^2)^{1/2}} \text{ for }i = k+1,..,n
\end{cases}
\end{equation}
Then the self shrinker equation \eqref{self shrinker equation} reduces to an ordinary differential equation
\begin{equation}
    \frac{u''}{1 + (u')^2} = (\frac{x}{2} - \frac{k-1}{x})u' - \frac{u}{2} + \frac{n-k}{u}\label{Main ODE y = u(x)}
\end{equation}
In case $x = v(y)$, same line of computation yields
\begin{equation}
    \frac{v''}{1 + (v')^2} = (\frac{y}{2} - \frac{n-k}{y})v' - \frac{v}{2} + \frac{k-1}{v}\label{Main ODE x = v(y)}
\end{equation}
Finally, we define some frequently used constants
\begin{equation}\label{frequently used constants}
    \alpha_{k} = \sqrt{2(k-1)}, \ \beta_{n,k} = \sqrt{2(n-k)}, \ \sigma_{n,k} = \frac{\sqrt{n-k}}{\sqrt{k-1}}. 
\end{equation}

\section{Existence of `trumpets' and its analysis}\label{section 3}
In this section, we prove theorem \ref{Main theorem : Existence of birotational trumpets}. The strategy is to construct a sequence of solutions to \eqref{Main ODE y = u(x)} with varying initial data. By showing that these approximate solutions obey certain uniform estimates, we take a subsequential limit which gives us the desired object. We note that the proof in proposition \ref{Existence of trumpets} is a slight modification of the proof of proposition 8.1 in \cite{Park}, but for the sake of completeness, we present the proof here as well.
\begin{proposition}\label{Existence of trumpets}
For each $2 \leq k \leq n-1$, $0 < \sigma < \sigma_{n,k}$, there exists $u_{\sigma} : [x_{s}(\sigma) , \infty) \to \mathbf{R}_+$ so that\\\\
(i) $x_{s}(\sigma) \geq \alpha_k$ and $u_{\sigma}(x_{s}(\sigma)) = \beta_{n,k}$.\\
(ii) $u_{\sigma} \in C^0([x_{s}(\sigma), \infty)) \cap C^{\infty}((x_{s}(\sigma), \infty))$  solves equation \eqref{Main ODE y = u(x)}. In other words, the surface $\Sigma^{k,n}_{\sigma}$ given by the parametrization
\begin{equation}
    [x_s(\sigma) , \infty) \times \mathbf{S}^{k-1}\times \mathbf{S}^{n-k} \ni (x, w_1, w_2) \to (xw_1, u_{\sigma}(x)w_2) \in  \mathbf{R}^{n+1}
\end{equation}
is a self-shrinker. \\
(iii) $u_{\sigma}$ is strictly increasing, and convex for $x > x_{s}(\sigma)$. Moreover, there exists $c = c(n,k) > 0$ so that
\begin{equation}\label{asymptotics of trumpets}
    |u_{\sigma}(x) - \sigma x| \leq \frac{c(n,k)}{\sigma x}, \ \ |u_{\sigma}'(x) - \sigma| \leq \frac{c(n,k)}{\sigma x^2}
\end{equation}
for all $x \geq x_{s}(\sigma)$. 
\end{proposition}
\begin{proof}[Proof of proposition \ref{Existence of trumpets}]
For each $0 < \sigma < \sigma_{n,k}$, $a > \frac{\beta_{n,k}}{\sigma}$, we define 
    \begin{equation}
        u_{\sigma, a} : (x_{0}(\sigma, a), a] \to \mathbf{R}_+
    \end{equation}
    to be the solution to the initial value problem
    \begin{equation}
        \begin{cases}
            \frac{u_{\sigma, a}''}{1 + (u_{\sigma, a}')^2} = (\frac{x}{2} - \frac{k-1}{x})u_{\sigma, a}' - \frac{u_{\sigma, a}}{2} + \frac{n - k}{u_{\sigma, a}} \\
            u_{\sigma, a}(a) = \sigma a \ \ \ \ u'_{\sigma, a}(a) = \sigma
        \end{cases}
    \end{equation}
    with 
    \begin{equation}\label{definition ofx0(sigma, a)}
      x_{0}(\sigma, a) = \inf\{x_0 \geq \alpha_k\ | \ u_{\sigma, a} \text{ is well defined on }(x_0, a], \ u_{\sigma, a} > \beta_{n,k} \}. 
    \end{equation}
 \textbf{Claim}: For any $x \in (x_0(\sigma, a), a]$,
\begin{equation}\label{key apriori estimate of appx sol for trumpets1}
       \max(\beta_{n,k}, \sigma x) \leq u_{\sigma, a}(x) < \sigma_{n,k}x,   
    \end{equation}
    and
    \begin{equation}\label{key apriori estimate of appx sol for trumpets2}
        \phi_{\sigma, a}(x) = (\frac{x}{2} - \frac{k-1}{x})u_{\sigma, a}' - \frac{u_{\sigma, a}}{2} + \frac{n - k}{u_{\sigma, a}} > 0.
    \end{equation}
    $u_{\sigma, a}(x) \geq \beta_{n,k}$ is automatic by the definition of $x_0(\sigma, a)$ given by \eqref{definition ofx0(sigma, a)}. When $x = a$, \eqref{key apriori estimate of appx sol for trumpets1}, \eqref{key apriori estimate of appx sol for trumpets2} hold by the initial condition $u_{\sigma, a}(a) = \sigma a, \ u_{\sigma, a}'(a)= \sigma$. By continuity of $u_{\sigma, a}$, there exists a small $\epsilon > 0$ so that  
   \begin{equation}
       u_{\sigma, a}(x) < \sigma_{n,k}x, \ \ \phi_{\sigma, a}(x) > 0
   \end{equation}
    for all $a - \epsilon \leq x \leq a$. Also, by \eqref{Main ODE y = u(x)}, we have
    \begin{equation}\label{relationship between two derivative and phi}
        u_{\sigma, a}'' = (1 + (u_{\sigma, a}')^2)\phi_{\sigma, a} > 0
    \end{equation}
    for $x \in [a - \epsilon, a)$. Integrating twice in $x$, together with the initial data, we obtain \begin{equation}
        \sigma x - \sigma a\leq u_{\sigma, a}(x) - u_{\sigma, a}(a) = u_{\sigma, a}(x) - \sigma a,
    \end{equation} hence
    \begin{equation}
        \sigma x \leq u_{\sigma,a}(x).
    \end{equation}
   Thus, \eqref{key apriori estimate of appx sol for trumpets1}, \eqref{key apriori estimate of appx sol for trumpets2} hold for all $x \in [a - \epsilon, a]$ for some $\epsilon > 0$.\\

Define 
\begin{equation}\label{definition of x0bar}
        \overline{x}_0(\sigma, a) = \inf\{x_0 > x_0(\sigma, a) \ | \ \text{\eqref{key apriori estimate of appx sol for trumpets1}, \eqref{key apriori estimate of appx sol for trumpets2} hold for } x \in [x_0, a]\}
    \end{equation}
    By previous discussion, $\overline{x_0}(\sigma, a) \in [x_0(\sigma, a), a)$ is well defined.\\

    To prove the claim, we show that 
    \begin{equation}
        \overline{x}_0(\sigma, a) = x_0(\sigma, a).
    \end{equation}
    Suppose not, and assume that $\overline{x}_0(\sigma, a) > x_0(\sigma, a)$. Then by the definition of $\overline{x}_0(\sigma, a)$, we must have
    \begin{equation}\label{Property 1}
        \text{\eqref{key apriori estimate of appx sol for trumpets1}, \eqref{key apriori estimate of appx sol for trumpets2} hold for $x \in (\overline{x_0}(\sigma, a), a]$},
    \end{equation}
     and either
     \begin{equation}\label{possible 2-1}
         u_{\sigma, a}(\overline{x}_0(\sigma, a)) = \sigma \overline{x}_0(\sigma, a),
     \end{equation}
      or 
      \begin{equation}\label{possible 2-2}
        u_{\sigma, a}(\overline{x}_0(\sigma, a)) = \sigma_{n,k}\overline{x}_0(\sigma, a) , 
      \end{equation}
      or 
      \begin{equation}\label{possible 2-3}
         \phi_{\sigma, a}(\overline{x}_0(\sigma, a)) = 0. 
      \end{equation}
      
      \eqref{possible 2-1} is not possible; otherwise by the mean value theorem, there exists $\overline{x}_0(\sigma, a) < x_1 < a$ so that $u_{\sigma, a}'(x_1) = \sigma$. On the other hand, by \eqref{Property 1} and \eqref{relationship between two derivative and phi}, $u_{\sigma, a}'$ is strictly increasing in $x \in (\overline{x_0}(\sigma, a), a]$. Since $u_{\sigma, a}'(a) = \sigma$, this implies that
      \begin{equation}
          \sigma = u_{\sigma, a}'(x_1) < u_{\sigma, a}'(a) = \sigma
      \end{equation}
      which is a contradiction.\\
      
      If \eqref{possible 2-2} holds, then we must have
      \begin{equation}\label{strict inequality in case2}
          u_{\sigma, a}'(\overline{x}_0(\sigma, a)) < \sigma_{n,k}.
      \end{equation}
     Indeed by \eqref{Property 1}, we have $u_{\sigma, a}'(\overline{x}_0(\sigma, a)) \leq \sigma_{n,k}$. If the strict inequality fails, then by uniqueness of solution to initial value problem, we must have
     \begin{equation}
         u_{\sigma, a}(x) = \sigma_{n,k}x
     \end{equation}
     for $x \geq \overline{x}_0(\sigma, a)$ which contradicts $u_{\sigma,a}'(a) = \sigma < \sigma_{n,k}$. Combining \eqref{strict inequality in case2} with the assumption $\overline{x_0}(\sigma, a) > x_0(\sigma, a) \geq  \alpha_k$, we obtain
\begin{align*}
     \phi_{\sigma, a}(\overline{x}_0(\sigma, a)) &< (\frac{\overline{x}_0(\sigma, a)}{2} - \frac{k-1}{\overline{x}_0(\sigma, a)})\sigma_{n,k} - \frac{\sigma_{n,k}\overline{x}_0(\sigma, a)}{2} + \frac{n-k}{\sigma_{n,k}\overline{x}_0(\sigma, a)} = 0.
\end{align*}
    On the other hand, by \eqref{Property 1}, and the continuity of $\phi_{\sigma, a}$, we must have $\phi_{\sigma, a}(\overline{x}_0(\sigma, a))\geq 0$ which is a contradiction. Therefore, we see that
    \begin{equation}\label{weak C^0 for trumpet}
        \sigma x \leq u_{\sigma, a}(x) < \sigma_{n,k}x
    \end{equation}
    hold for $x \in [ \overline{x}_0(\sigma, a), a]$.\\
    
    If \eqref{possible 2-3} holds, then by \eqref{Property 1}, we must have
    \begin{equation}
        \phi_{\sigma, a}'(\overline{x}_0(\sigma, a)) \geq 0.
    \end{equation}
    On the other hand, $\phi_{\sigma, a}(\overline{x}_0(\sigma, a)) = 0$ together with the assumptions $\overline{x}_0(\sigma, a) > x_0(\sigma, a) \geq \alpha_k$, we have
    $$u_{\sigma, a}'(\overline{x}_0(\sigma, a)) = \frac{\overline{x}_0(\sigma, a)(u^2_{\sigma, a} - \beta_{n,k}^2)}{u_{\sigma, a}(\overline{x}_0^2(\sigma, a) - \alpha_k^2)}> 0.$$
    By using \eqref{Main ODE y = u(x)} to compute $\phi_{\sigma, a}'$, and then using \eqref{Property 1}, \eqref{weak C^0 for trumpet}, and $u'_{\sigma, a} >0$, we have
\begin{equation}
    \phi_{\sigma, a}'(\overline{x_0}(\sigma, a)) = [\frac{k-1}{\overline{x_0}(\sigma, a)^2} - \frac{n - k}{u_{\sigma, a}(\overline{x_0}(\sigma, a))^2}]u_{\sigma, a}'(\overline{x_0}(\sigma, a)) < 0
\end{equation}
    which is a contradiction. \\
    
    Since all three cases \eqref{possible 2-1}, \eqref{possible 2-2}, \eqref{possible 2-3} are impossible, the initial assumption $\overline{x}_0(\sigma, a) > x_0(\sigma, a)$ has to be  false. Thus we obtain
 \begin{equation}
     \overline{x}_0(\sigma, a) =x_0(\sigma, a),
 \end{equation}
    thus proving the claim.\\

   We now use \eqref{key apriori estimate of appx sol for trumpets1}, \eqref{key apriori estimate of appx sol for trumpets2} to pass through the limit as $a \to \infty$.
  \eqref{key apriori estimate of appx sol for trumpets1}, \eqref{key apriori estimate of appx sol for trumpets2}, and \eqref{relationship between two derivative and phi} implies that $ 0 \leq u_{\sigma, a}' \leq \sigma$, $u_{\sigma}'' \geq 0$, and $u_{\sigma}$ is bounded from below by $\beta_{n,k}$. Hence $u_{\sigma}$ extends in $C^1$ to $x = x_0(\sigma, a)$. We also obtain
\begin{equation}\label{estimate of x0(sigma, a)}
      x_0(\sigma, a) \in [\alpha_k, \frac{\beta_{n,k}}{\sigma}], \ \ \ u_{\sigma, a}(x_{0}(\sigma, a)) = \beta_{n,k}.
 \end{equation}
To see this, it is enough to show that $u_{\sigma, a}(x_0(\sigma, a)) = \beta_{n,k}$ since by \eqref{definition ofx0(sigma, a)}, \eqref{key apriori estimate of appx sol for trumpets1}, we have $\alpha_k \leq x_0(\sigma, a)$, and $\sigma x_0(\sigma, a) \leq \beta_{n,k}$. If $u_{\sigma, a}(x_0(\sigma, a)) > \beta_{n,k}$, then by \eqref{key apriori estimate of appx sol for trumpets1}, $x_0(\sigma, a) > \alpha_k$, hence
 \begin{equation}
     \lim_{x \to x_0(\sigma, a)}(x, u_{\sigma, a}(x)) \in (\alpha_k, \infty) \times (\beta_{n,k}, \infty).
 \end{equation}
    Then in view of the ODE \eqref{Main ODE y = u(x)}, one can strictly extend the graphical solution to parts of $\alpha_k < x < x_0(\sigma, a)$, and still retain $u_{\sigma, a} > \beta_{n,k}$. This contradicts the minimality of $x_0(\sigma, a)$ given by \eqref{definition ofx0(sigma, a)}. \\

  Discussions so far implies that we have
    \begin{equation}\label{apriori estimate for trumpets1}
        \max(\sigma x, \beta_{n,k}) \leq u_{\sigma, a}(x) <\sigma_{n,k}x, \ 0 \leq u_{\sigma, a}'(x) \leq \sigma, \  \frac{u_{\sigma, a}''}{1 + (u_{\sigma, a}')^2} \geq 0
    \end{equation}
    for $x \in (x_0(\sigma, a), a]$, and
    \begin{equation}
        \alpha_k \leq x_0(\sigma, a) \leq \frac{\beta_{n,k}}{\sigma}, \ u_{\sigma, a}(x_0(\sigma, a)) = \beta_{n,k}.
    \end{equation}
    Then one can find $a_j \to \infty$ so that
    $$x_0(\sigma, a_j) \to x_{s}(\sigma) \in [\alpha_k, \frac{\beta_{n,k}}{\sigma}]$$ and 
    $$u_{\sigma, a_j} \to u_{\sigma} \text{ in }C^2_{loc}((x_{s}(\sigma), \infty)).$$
    We thus see that $u_{\sigma}$ solves \eqref{Main ODE y = u(x)}. Also, by \eqref{apriori estimate for trumpets1}, $u_{\sigma}$ satisfies
   \begin{equation}\label{final apriori estimates for trumpets}
       \max(\beta_{n,k}, \sigma x) \leq u_{\sigma}(x) \leq \sigma_{n,k}x, \   0 \leq u_{\sigma}'(x) \leq \sigma, \  \frac{u_{\sigma}''}{1 + (u_{\sigma}')^2}\geq 0
   \end{equation}
   for all $x \in (x_{s}(\sigma), \infty)$. \\
   
   To obtain boundary information, first note that \eqref{final apriori estimates for trumpets} implies that $u_{\sigma}$ extends as a $C^1$ function up to $x = x_{s}(\sigma)$ with $u_{\sigma}(x_{s}(\sigma)) \geq \beta_{n,k}$. We claim that 
   $$u_{\sigma}(x_{s}(\sigma)) = \beta_{n,k}$$
   If $u_{\sigma}(x_{s}(\sigma)) > \beta_{n,k}$, then by $0 \leq u_{\sigma}'(x) \leq \sigma$, we can find some small $\delta > 0$ so that 
   $u_{\sigma}(x) \geq \beta_{n,k} + \delta$
for all $x \geq x_s(\sigma)$. On the other hand, since $x_0(\sigma, a_j) \to x_{s}(\sigma)$ with $u_{\sigma, a}(x_0(\sigma, a)) = \beta_{n,k}$ and $0 \leq u_{\sigma, a}'(x) \leq \sigma$, we have for all sufficiently large $j$, 
\begin{equation}
    u_{\sigma, a_j}(x_{s}(\sigma) + \frac{\delta}{100\sigma}) \leq \beta_{n,k} + \frac{\delta}{50}
\end{equation}
which contradicts $u_{\sigma, a_j} \to u_{\sigma}$ in $C^2_{loc}((x_s(\sigma), \infty))$.
This proves item (i), (ii) in proposition \ref{Existence of trumpets}. \\

We now prove item (iii). Convexity is immediate from \eqref{final apriori estimates for trumpets}.
To prove that $u_{\sigma}$ is strictly increasing in $x$, 
we first show that 
\begin{equation}
    u_{\sigma}(x) > \beta_{n,k}
\end{equation}
for all $x > x_s(\sigma)$. Suppose this is not the case, and there exists $x_1 > x_s(\sigma)$ so that $u_{\sigma}(x_1) = \beta_{n,k}$. Then by \eqref{final apriori estimates for trumpets}, we must have 
\begin{equation}
    u_{\sigma}(x_s(\sigma)) = \beta_{n,k}, \ u_{\sigma}'(x_s(\sigma)) = 0.
\end{equation}
By uniqueness of solutions to initial value problems, this implies that 
\begin{equation}
    u_{\sigma}(x) \equiv \beta_{n,k}
\end{equation}
for all $x \geq x_s(\sigma)$, which contradicts \eqref{final apriori estimates for trumpets}. Once we have $u_{\sigma}(x) > \beta_{n,k}$ for all $x > x_s(\sigma)$, we can use \eqref{final apriori estimates for trumpets} to obtain
\begin{equation}
    u_{\sigma}'(x) \geq (\frac{x}{2} - \frac{k-1}{x})^{-1}(\frac{u_{\sigma}(x)}{2} - \frac{n-k}{u_{\sigma}(x)}) > 0,
\end{equation}
which implies that $u_{\sigma}$ is strictly increasing. \\

To show the desired asymptotics of $u_{\sigma}$, we go back to the approximate solutions $u_{\sigma, a}$. By following the proof of lemma 2 of \cite{KM}, each $u_{\sigma, a}$ satisfies the integral identity
    $$u_{\sigma, a}(x) = \sigma x + x\int_{x}^{a}\frac{1}{t^2}\int_{t}^{a}(\frac{(n-k)s}{u_{\sigma, a}} - (k-1)u_{\sigma, a}')(1 + (u_{\sigma, a}')^2)e^{-\frac{1}{2}\int_{t}^{s}z(1 + (u_{\sigma, a}')^2)dz}dsdt.$$
     \eqref{apriori estimate for trumpets1} implies that above integral identity also passes through the limit and thus
 \begin{align*}
     u_{\sigma}(x) = \sigma x + x\int_{x}^{\infty}\frac{1}{t^2}\int_{t}^{\infty}(\frac{(n-k)s}{u_{\sigma}} - (k-1)u_{\sigma}')(1 + (u_{\sigma}')^2)e^{-\frac{1}{2}\int_{t}^{s}z(1 + (u_{\sigma}')^2)dz}dsdt
 \end{align*}
    holds for all $x\in [x_{s}(\sigma), \infty)$. Then by using estimates \eqref{final apriori estimates for trumpets}, we have 
 \begin{align*}
     |u_{\sigma}(x) - \sigma x| \leq x\int_{x}^{\infty}\frac{1}{t^2}\int_{t}^{\infty}\frac{n-k}{\sigma}(1 + \frac{n-k}{k-1})e^{-1/4(s^2 - t^2)}dsdt \leq \frac{c(n,k)}{\sigma x}.
 \end{align*}
    Similar calculations yield 
    \begin{equation}
        |u_{\sigma}'(x) - \sigma| \leq \frac{c(n,k)}{\sigma x^2}.
    \end{equation}
    This completes the proof of proposition \ref{Existence of trumpets}.
\end{proof}
\begin{remark}\label{smooth dependence of trumpets on slope}
    The asymptotics of $u_{\sigma}$ in (iii) of proposition \ref{Existence of trumpets} is exactly the same as the $C^1$-estimates for Kleene-M\o ller trumpets given in lemma 3 of \cite{KM}. Also, by following the calculations of lemma 4 in \cite{KM}, we see that $u_{\sigma}$ can be obtained as a fixed point of a contraction map given by the integral identity in a suitable Banach space at least for large $x$.  This implies that $u_{\sigma}$ depends smoothly on $\sigma$.
\end{remark}
We now turn to obtaining a more precise formula for $u_{\sigma}$. 
In (iii) of proposition \ref{Existence of trumpets}, we obtained a $C^0$-estimate which is satisfactory for large enough $x$ for each fixed $\sigma$. However, it is not desirable when analyzing how the `trumpets' behave for different values of $\sigma$ at each $x$ due to the $\frac{1}{\sigma}$ in the right hand side. In this section, we prove an alternative $C^0$-estimate which is more useful in this perspective. The idea of the proof is to obtain a pointwise estimate for $u_{\sigma}'$ which can be integrated in $x$ to obtain the desired $C^0$-estimate. 
\begin{lemma}\label{Alternative $C^0$ for trumpets}
For each $2 \leq k \leq n-1$, there exists $l_0 = l_0(n,k) > \alpha_k$, and $C = C(n,k) > 0$ so that 
   \begin{align*}
       \max(\sigma x, \sqrt{\beta_{n,k}^2 + \sigma^2e^{-C/x^2}(x^2 - \alpha_k^2})) \leq u_{\sigma}(x) \leq  \sqrt{\beta_{n,k}^2 + \sigma^2(x^2 - \alpha_k^2)}
   \end{align*}
for all $\sigma \in (0, \sigma_{n,k})$, $x \in [\max(x_{s}(\sigma),l_0(n,k)), \infty)$.
\end{lemma}
We first mention two immediate consequences of lemma \ref{Alternative $C^0$ for trumpets}. The first corollary gives a uniform upper bound for $x_{s}(\sigma) > 0$ independent of $\sigma$.
\begin{corollary}\label{Upper bound of starting point for trumpets}
    For $l_0 = l_0(n,k)$ given by lemma \ref{Alternative $C^0$ for trumpets}, $x_{s}(\sigma) < l_0$ for all $\sigma \in (0, \sigma_{n,k})$. In other words, the `trumpets' can be extended all the way up to some $l_0$ which is independent of $\sigma$.
\end{corollary}
\begin{proof}[Proof of corollary \ref{Upper bound of starting point for trumpets}]
   If $\alpha_k < l_0 \leq x_s(\sigma)$ for some $\sigma \in (0, \sigma_{n,k})$, then by proposition \ref{Existence of trumpets} and proposition \ref{Alternative $C^0$ for trumpets}, we must have
   \begin{equation}
        \beta_{n,k} < \sqrt{\beta^2_{n,k} + \sigma^2e^{-C/x_s(\sigma)^2}(x_s(\sigma)^2 - \alpha_k^2})\leq u_{\sigma}(x_s(\sigma)) = \beta_{n,k}
   \end{equation}
   which is a contradiction.
\end{proof}
The second corollary implies that the family of `trumpets' indeed sweep the region between a generalized cylinder $\mathbf{R}^k\times \mathbf{S}^{n-k}(\sqrt{2(n-k)})$, and the $O(k)\times O(n-k+1)$-symmetric minimal cone outside some compact set containing the origin.
\begin{corollary}\label{Sweeping out by trumpets}
For any $(x, y) \in (l_0, \infty) \times (\beta_{n,k}, \infty)$ with $y < \sigma_{n,k}x$, there exists $\sigma \in (0, \sigma_{n,k})$ so that
    $$y = u_{\sigma}(x).$$
    Here, $l_0 > 0$ is the constant from corollary \ref{Upper bound of starting point for trumpets}. 
\end{corollary}
\begin{proof}[Proof of corollary \ref{Sweeping out by trumpets}]
    One can find $0 < \sigma_0 < \sigma_1 < \sigma_{n,k}$ so that
    $$ u_{\sigma_0}(x) \leq\sqrt{\beta_{n,k}^2 +  \sigma_0^2x^2- \sigma_0^2\alpha_k^2} < y < \sigma_1x \leq u_{\sigma_1}(x)$$
    where the first and last inequalities are due to lemma \ref{Alternative $C^0$ for trumpets}. 
    Since $x_{s}(\sigma) < l$ for all $\sigma \in (0, \sigma_{n,k})$, remark \ref{smooth dependence of trumpets on slope} implies that $\sigma \to u_{\sigma}(x)$ is smooth, hence by intermediate value theorem, one can find $\sigma \in (0, \sigma_{n,k})$ so that $$y=u_{\sigma}(x).$$
\end{proof}
\begin{proof}[Proof of lemma \ref{Alternative $C^0$ for trumpets}]
    By continuity of $u_{\sigma}$ with respect to $x$, it is enough to consider when $x > \max(x_{s}(\sigma), l_0)$ for $l_0 > \alpha_k$ to be determined later. Recall item (iii) in proposition \ref{Existence of trumpets}, we have $u''_{\sigma}(x) \geq 0$ for all $x \in (x_{s}(\sigma), \infty)$. By using \eqref{Main ODE y = u(x)}, this implies
\begin{equation}
    \frac{d}{dx}\ln(u_{\sigma}^2 - \beta_{n,k}^2) \geq \frac{d}{dx}\ln(x^2 - \alpha_k^2).
\end{equation}
   Integrating from each $x \in (x_{s}(\sigma), \infty)$ to any large $a > x$, and letting $a \to \infty$ (here, we use the asymptotic formula in item (iii) in proposition \ref{Existence of trumpets}) gives us the desired upper bound
    \begin{equation}
        u_{\sigma}(x) \leq \sqrt{\beta_{n,k}^2 + \sigma^2(x^2 - \alpha_k^2)}.
    \end{equation}
    We now focus on the lower bound. Since $\sigma x \leq u_{\sigma}(x)$ is proved in \eqref{final apriori estimates for trumpets}, we only need to show
    \begin{equation}
        u_{\sigma}^2(x) \geq \beta_{n,k}^2 + \sigma^2e^{-C/x^2}(x^2 - \alpha_k^2).
    \end{equation}
    The proof is motivated by proof of proposition 8.10 of \cite{ADS}. For each $\sigma \in (0, \sigma_{n,k})$, define
\begin{equation}
    w(x) = w_{\sigma}(x) = (\frac{x^2 - \alpha_k^2}{2x})\frac{d}{dx}\ln(u_{\sigma}^2(x) - \beta_{n,k}^2).
\end{equation}
Then by direct computation using \eqref{Main ODE y = u(x)}, $w_{\sigma}$ solves the equation
\begin{equation}\label{ODE for w in trumpet}
    (\frac{x}{2} - \frac{k-1}{x})w_{\sigma}'(x) = Lw_{\sigma}
\end{equation}
with
\begin{align*}
     Lw_{\sigma} = (\frac{1}{2} + \frac{k-1}{x^2})w_{\sigma} - (\frac{1}{2} + \frac{n-k}{u_{\sigma}^2})w_{\sigma}^2 + (\frac{x}{2} - \frac{k-1}{x})^2(1 + (u_{\sigma}')^2)(w_{\sigma} - 1).
\end{align*}
We now obtain an estimate for $w(x)$.
\begin{lemma}\label{Derivative bound for trumpets}
    There exists $l_0 = l_0(n, k) > 0$ so that
    \begin{equation}
        1 \leq w_{\sigma}(x) \leq 1 + \frac{72}{x^2}
    \end{equation}
    for all $x >\max(x_0, x_{s}(\sigma))$.
\end{lemma}
\begin{proof}[Proof of lemma \ref{Derivative bound for trumpets}]
 By $u_{\sigma}'' \geq 0$ and $u_{\sigma} > \beta_{n,k}$, we have the desired lower bound $w_{\sigma} \geq 1$. \\
 
 To derive the upper bound, we first choose $l_0$ so that
\begin{equation}\label{condition on x1}
    \frac{x}{2} - \frac{k-1}{x }\geq \frac{x}{3}
\end{equation}
for all $x \geq l_0$. We claim that for each $0 <\delta <1$
    \begin{equation}\label{axilary lemma}
        w_{\sigma}(x) \leq 1 + \delta
    \end{equation}
    for all $x > \max(x_{s}(\sigma), l_0,   \frac{6}{\sqrt{\delta}})$. Let 
    \begin{equation}
        p_0(\delta) = \sup\{x > \max(x_{s}(\sigma), l_0) \ | \ w_{\sigma}(x) > 1 + \delta \}.
    \end{equation}
    If the set is empry, then \eqref{axilary lemma} follows trivially, so we assume that the set is nonempty. In such case, the asymptotics of $u_{\sigma}$ given by proposition \ref{Existence of trumpets} implies that
    \begin{equation}
        \lim_{x \to \infty}w_{\sigma}(x) = 1,
    \end{equation}
    hence $p_0(\delta) < \infty$.  Then one has
   $$w_{\sigma}(p_0(\delta)) = 1 + \delta, \ \ w_{\sigma}'(p_{0}(\delta)) \leq 0.$$
   By evaluating \eqref{ODE for w in trumpet} at $x = p_0(\delta)$, and noting that $u_{\sigma} \geq \beta_{n,k}$, $p_0(\delta) \geq l_0$, we have 
   $$0 \geq -(1 + \delta)^2 + \frac{p_0(\delta)^2}{9}\delta$$
   which implies that
   \begin{equation}
       p_0(\delta) \leq \frac{6}{\sqrt{\delta}}.
   \end{equation}
   Thus if $x > \frac{6}{\sqrt{\delta}} \geq p_0(\delta)$, then
   $$w_{\sigma}(x) \leq 1 + \delta,$$
   which is precisely \eqref{axilary lemma}. \\

Lemma \eqref{axilary lemma} immediately follows. Indeed, by possibly enlarging $l_0$, we may assume that
\begin{equation}
    1 + \frac{72}{x^2} < 2
\end{equation}
for all $x \geq l_0$. If lemma \eqref{axilary lemma} is false, then there exists $\overline{x} \geq l_0$ so that
\begin{equation}
    w_{\sigma}(\overline{x}) > 1 + \frac{72}{\overline{x}^2}.
\end{equation}
Set $\delta = \frac{72}{\overline{x}^2} \in (0,1)$. Then we see that 
\begin{equation}
    w_{\sigma}(\frac{6\sqrt{2}}{\sqrt{\delta}}) > 1 + \delta
\end{equation}
which is a contradiction.
\end{proof}

We can now complete the proof of lemma \ref{Alternative $C^0$ for trumpets} by first integrating lemma \ref{Derivative bound for trumpets} from each $x > \max(x_{s}(\sigma), l_0)$ to any large $a > x$. This implies
\begin{equation}
    \frac{u_{\sigma}(a)^2 - \beta_{n,k}^2}{u_{\sigma}^2(x) - \beta_{n,k}^2} \leq \frac{a^2 - \alpha_k^2}{x^2 - \alpha_k^2}e^{C(1/x^2 - 1/a^2)}.
\end{equation}
Letting $a \to \infty$ and using the asymptotics of $u_{\sigma}$ given by proposition \ref{Existence of trumpets}, we obtain the desired lower bound
$$u_{\sigma}^2(x) \geq \beta_{n,k}^2 + \sigma^2e^{-C/x^2}(x^2 - \alpha_k^2)$$
for all $x > \max(l_0, x_{s}(\sigma))$. This completes the proof of lemma \ref{Alternative $C^0$ for trumpets}.
\end{proof}
\section{Existence of `disks' and its analysis}\label{section 4}
In this section, we prove the existence of `disks' in theorem \ref{Main theorem : existence of birotational caps}. This can be done by finding some continuous function $u : [x_{s}(a), a] \to \mathbf{R}_+$ which solves equation \eqref{Main ODE y = u(x)} in $(x_{s}(a), a)$, and satisfies
$$u(a) = 0 , \ \lim_{x \to a-}u'(x) = -\infty$$
for some $a > 0$. The strategy is the same as that in proposition \ref{Existence of trumpets}, except that we look at its inverse $x = v(y)$. We will achieve the desired $v$ by constructing a sequence of solutions to equation \eqref{Main ODE x = v(y)} with varying initial data. Then by showing that they obey some uniform estimates, we construct the desired self shrinker as a limit of the sequence of solutions. 
\begin{proposition}\label{Existence of caps}
    For each $2 \leq k \leq n-1$, $a > \alpha_k$, there exists $u_a : [x_{s}(a), a] \to \mathbf{R}_+$ so that\\\\
    (i) $x_{s}(a) \geq \alpha_k$ and $u_a(x_{s}(a)) = \beta_{n,k}$. \\
    (ii) $u_{a} \in C^0([x_{s}(a), a]) \cap C^{\infty}((x_{s}(a), a))$ solves equation \eqref{Main ODE y = u(x)} in $(x_{s}(a), a)$. Moreover, $u_a$ is concave, strictly decreasing and satisfy
\begin{equation}
    u_a(a) = 0, \ \lim_{x \to a-}u_{a}'(x) = -\infty.
\end{equation}
Hence, the surface $\Gamma^{n,k}_a$ given by
\begin{equation}
    [x_s(a), a]\times \mathbf{S}^{k-1}\times \mathbf{S}^{n-k} \ni (x, w_1, w_2) \to (xw_1, u_a(x)w_2) \in  \mathbf{R}^{n+1}
\end{equation}
is a self-shrinker.\\
(iii) 
We have a $C^0$-estimate
\begin{equation}\label{C0 estimate for caps in proposisition 4.1}
    \beta_{n,k}\sqrt{1 - \frac{x^2 - \alpha_k^2}{a^2 - \alpha_k^2}} \leq u_a \leq \beta_{n,k},
\end{equation}
and a $C^1$-estimate
\begin{equation}\label{C1 estimate for caps in proposisition 4.1}
     u_a'(x) \leq  (\frac{u_a(x)}{2} - \frac{\beta_{n,k}^2}{2u_a(x)})\frac{2x}{x^2 - \alpha_k^2}.
\end{equation}
\end{proposition}
\begin{proof}[Proof of proposition \ref{Existence of caps}]
    The proof is similar to that of proposition \ref{Existence of trumpets}. It turns out that it is easier to find `disks' by viewing the profile curve as a graph over $y$ axis. For each small $\epsilon > 0$, we consider $v_{a, \epsilon}: [\epsilon, y_0(\epsilon, a)) \to \mathbf{R}_+$ which solves the initial value problem
    \begin{equation}
        \begin{cases}
            \frac{v_{a, \epsilon}''}{1 + (v_{a, \epsilon}')^2} = (\frac{y}{2} - \frac{n - k}{y})v_{a, \epsilon}' - \frac{v_{a, \epsilon}}{2} + \frac{k-1}{v_{a, \epsilon}} \\
            v_{a, \epsilon}(\epsilon) = a \ \ \ v_{a, \epsilon}'(\epsilon) = 0
        \end{cases}
    \end{equation}
with
\begin{equation}\label{definition of y0(epsilon, a)}
    y_0(\epsilon,a) = \sup \{y_0 \leq \beta_{n,k}\ | \ v_{a, \epsilon} \text{ is well defined in }[\epsilon, y_0) \}.
\end{equation}
 \textbf{Claim} For any $y \in [\epsilon, y_0(\epsilon, a))$,
 \begin{equation}\label{key apriori estimates for appx solution for disks}
     \alpha_k < v_{a, \epsilon}(y) \leq a, \ \phi_{a, \epsilon}(y) = (\frac{y}{2} - \frac{n-k}{y})v_{a, \epsilon}'(y) - \frac{v_{a, \epsilon}}{2} + \frac{k-1}{v_{a, \epsilon}} < 0
 \end{equation}
holds. \\

When $y = \epsilon$, the given initial data implies that \eqref{key apriori estimates for appx solution for disks} hold. Then by continuity of $v_{a, \epsilon}$, there exists small $\delta > 0$ so that 
$$\alpha_k < v_{a, \epsilon}(y), \ \phi_{a, \epsilon}(y) < 0$$
hold for $y \in [\epsilon, \epsilon + \delta]$. 
Then by equation \eqref{Main ODE x = v(y)}, we have \begin{equation}
    v_{a, \epsilon}''  = (1 + (v_{a, \epsilon}')^2)\phi_{a, \epsilon} < 0
\end{equation} for $y \in (\epsilon, \epsilon + \delta]$. Combining with $v_{a, \epsilon}'(\epsilon) = 0$, we have 
$$v_{a, \epsilon}' \leq 0$$
for $y \in [\epsilon, \epsilon + \delta]$. Thus combining with $v_{a, \epsilon}(\epsilon) = a$, we see that $v_{a, \epsilon}(y) \leq a$ for $y \in [\epsilon, \epsilon + \delta]$, hence \eqref{key apriori estimates for appx solution for disks} hold for all $y \in [\epsilon, \epsilon + \delta]$ for some small $\delta > 0$.\\

Let 
\begin{equation}\label{definition of overline y0(epsilon, a)}
    \overline{y}_0(\epsilon, a) = \sup\{y_0 < y_0(\epsilon, a)\ | \ \text{\eqref{key apriori estimates for appx solution for disks} hold for }y \in [\epsilon, y_0] \}. 
\end{equation}
Above discussion implies that $\overline{y}_0(\epsilon, a) \in (\epsilon, y_0(\epsilon, a)]$ is well defined. \\

To prove \eqref{key apriori estimates for appx solution for disks} for all $y \in [\epsilon, y_0(\epsilon, a))$, we need to show that $\overline{y}_0(\epsilon, a) = y_0(\epsilon, a)$. Suppose not, and assume $\overline{y}_0(\epsilon, a) < y_0(\epsilon, a)$. By the definition of $\overline{y}_0(\epsilon, a)$ given in \eqref{definition of overline y0(epsilon, a)}, we see that
\begin{equation}\label{property 1 for disks}
    \text{\eqref{key apriori estimates for appx solution for disks} hold for }y \in [\epsilon, \overline{y}_0(\epsilon, a)),
\end{equation}
and either
\begin{equation}\label{possible 2-1 for caps}
    v_{a, \epsilon}(\overline{y}_0) = a,
\end{equation}
or 
\begin{equation}\label{possible 2-2 for caps}
   v_{a, \epsilon}(\overline{y}_0) = \alpha_k, 
\end{equation}
or 
\begin{equation}\label{possible 2-3 for caps}
\phi_{a, \epsilon}(\overline{y}_0) = 0.    
\end{equation}

If \eqref{possible 2-1 for caps} holds, by mean value theorem, there exists $y_1 \in (\epsilon, \overline{y}_0(\epsilon, a)) $ so that $v_{a, \epsilon}'(y_1) = 0$. On the other hand, by \eqref{property 1 for disks} and equation \eqref{Main ODE x = v(y)}, we see that $v_{a, \epsilon}'$ is strictly decreasing for $y \in [\epsilon, \overline{y}_0(\epsilon,a))$. Combining with the initial condition, we see that
\begin{equation}
    0 = v_{a, \epsilon}'(y_1) < v_{a, \epsilon}'(\epsilon) = 0
\end{equation}
which is a contradiction.\\

Assume \eqref{possible 2-2 for caps} holds. Then by \eqref{property 1 for disks}, $v_{a, \epsilon}'(\overline{y}_0) \leq 0$. Actually, we have a strict inequality
\begin{equation}
    v_{a, \epsilon}'(\overline{y}_0) < 0.
\end{equation}
If not, then by the uniqueness of solution to initial value problem, we must have
 $$v_{a, \epsilon} = \alpha_k$$
 for $y \in [\epsilon, \overline{y}_0(\epsilon, a)]$ which contradicts $v_{a, \epsilon}(\epsilon) = a > \alpha_k$. 
 Then by $\overline{y}_0 < y_0(\epsilon, a) \leq \beta_{n,k}$,
\begin{equation}
    \phi_{a, \epsilon}(\overline{y}_0(\epsilon, a)) > (\frac{\overline{y}_0(\epsilon, a)}{2} - \frac{n-k}{\overline{y}_0(\epsilon, a)})\times 0 - \frac{\alpha_k{}}{2} + \frac{k-1}{\alpha_k} =0
\end{equation}
which contradicts \eqref{property 1 for disks} by continuity of $\phi_{a, \epsilon}$. Therefore, we see that 
\begin{equation}
    \alpha_k < v_{a, \epsilon}(y) \leq a
\end{equation}
holds for all $y \in [\epsilon, \overline{y}_0(\epsilon, a)]$.\\

Suppose \eqref{possible 2-3 for caps} holds. Then by \eqref{property 1 for disks}, $\phi_{a, \epsilon}(\overline{y}_0(\epsilon, a)) = 0$ and $\phi_{a, \epsilon}'(\overline{y}_0) \geq 0$. On the other hand, by solving $\phi_{a, \epsilon}(\overline{y}_0(\epsilon, a)) = 0$ for $v_{a, \epsilon}'(\overline{y}_0(\epsilon, a))$, and using the assumption $\overline{y}_0(\epsilon, a) < \beta_{n,k}$ and the fact that $v_{a, \epsilon}(\overline{y}_0(\epsilon, a)) > \alpha_k$, we have
$$v_{a, \epsilon}'(\overline{y}_0(\epsilon, a)) = (\frac{\overline{y}_0(\epsilon, a)}{2} - \frac{n-k}{\overline{y}_0(\epsilon, a)})^{-1}(\frac{v_{a, \epsilon}(\overline{y}_0(\epsilon, a))}{2} - \frac{k-1}{v_{a, \epsilon}(\overline{y}_0(\epsilon, a))}) < 0.$$
By using equation \eqref{Main ODE x = v(y)} to compute $\phi_{a, \epsilon}'(\overline{y}_0(\epsilon, a))$, and using $\overline{y}_0(\epsilon, a) < \beta_{n,k}$ and $v_{a, \epsilon}(\overline{y}_0(\epsilon, a)) > \alpha_k$, we see that
\begin{equation}
  \phi_{a, \epsilon} '(\overline{y}_0(\epsilon, a)) = (\frac{n - k}{\overline{y}_0(\epsilon, a)^2} - \frac{k-1}{v_{a, \epsilon}^2(\overline{y}_0(\epsilon, a))})v_{a, \epsilon}'(\overline{y}_0(\epsilon, a)) < 0  
\end{equation}
which is a contradiction. \\

Since all three possibilities lead to a contradiction, we must have
$$\overline{y}_0(\epsilon, a) = y_0(\epsilon, a),$$
thus proving the claim.\\

As a direct consequence, we have $y_0(\epsilon, a) = \beta_{n,k}$. By estimate \eqref{key apriori estimates for appx solution for disks}, 
we have 
\begin{equation}\label{a priori estimate for appx solution for caps}
    \alpha_k < v_{a, \epsilon}(y) \leq a, \ (\frac{a}{2} - \frac{k-1}{a})(\frac{y}{2} - \frac{n-k}{y})^{-1} \leq v_{a, \epsilon}'(y) \leq 0, \ v_{a, \epsilon}'' < 0
\end{equation}
for all $y \in [\epsilon, y_0(\epsilon, a))$. Thus
if $y_0(\epsilon, a) < \beta_{n,k}$, by estimates \eqref{a priori estimate for appx solution for caps}, we can slightly extend $v_{a, \epsilon}$ beyond $y_0(\epsilon, a)$, which contradicts the maximality of $y_0(\epsilon, a)$ given by \eqref{definition of y0(epsilon, a)}.\\

Also, we can obtain a better $C^0$-estimate by the derivative estimate in \eqref{a priori estimate for appx solution for caps}. By rearranging terms in the derivative estimate, we have 
$$ \frac{d}{dy}\ln(\beta_{n,k}^2 - y^2)< \frac{d}{dy}\ln(v_{a, \epsilon}^2 - \alpha_k^2).$$
By integrating from $\epsilon$ to any $\epsilon < y < y_0(\epsilon, a)$, and using the initial condition $v_{a, \epsilon}(\epsilon) = a$, we have
\begin{equation}\label{better c0 estimate for appx solution for cap}
     \alpha_k^2 + (a^2 - \alpha_k^2)\frac{\beta_{n,k}^2 - y^2}{\beta_{n,k}^2 - \epsilon^2}< v_{a, \epsilon}^2 \leq a^2.
\end{equation}

The estimates \eqref{a priori estimate for appx solution for caps} together with the fact that $y_0(\epsilon, a) = \beta_{n,k}$ allow us to find $\epsilon_j \to 0$ and obtain a $C^2_{loc}((0, \beta_{n,k}))$ limit denoted by $v_a$. This limit solves equation \eqref{Main ODE x = v(y)} for $y \in (0, \beta_{n,k})$. Also, the estimates \eqref{a priori estimate for appx solution for caps}, \eqref{better c0 estimate for appx solution for cap} passes through the limit, thus we have the $C^0$-estimate
\begin{equation}\label{final c0 estimate for cap}
    \sqrt{\alpha_k^2 + (a^2 - \alpha_k^2)(1 - \frac{y^2}{\beta_{n,k}^2})} \leq v_a(y) \leq a
\end{equation}
and the $C^1$-estimate
\begin{equation}\label{final c1 estimate for cap}
   (\frac{y}{2} - \frac{\beta_{n,k}^2}{2y})^{-1}(\frac{v_a}{2} - \frac{\alpha_k^2}{2v_a}) \leq v_a'(y) \leq 0 
\end{equation}
for all $y \in (0, \beta_{n,k})$. Above estimates imply that $v_a$ extend up to $y = 0$ in $C^1$ with 
\begin{equation}\label{initial condition of va cap}
    v_a(0) = a, \ v_a'(0) = 0
\end{equation}
and extends continuously up to $y = \beta_{n,k}$ in $C^0$ with $v_a(\beta_{n,k}) \geq \alpha_k$. \\

By \eqref{final c1 estimate for cap} and equation \eqref{Main ODE x = v(y)}, $v_a'' \leq 0$ and $v_a' \leq 0$. $v_a$ is strictly decreasing in $y$ because if this is not the case, then one can find some interval $[\alpha, \beta] \subset [0, \beta_{n,k}]$, so that $v_a$ is constant on that interval. By the estimate \eqref{final c0 estimate for cap}, that constant $u_{a}(\alpha)$ must be strictly larger than $\alpha_k$. However, equation \eqref{Main ODE x = v(y)} does not hold for constant function $x = u_{a}(\alpha) > \alpha_k$ which is a contradiction. The fact that $v_a$ strictly decreases in $y$ together with $v_a'(0) = 0$, $v_a'' \leq 0$ implies that $v_a' < 0$ for all $y \in (0, \beta_{n,k})$.\\

Above discussion implies that by inverse function theorem, we can find the inverse $y = u_a(x)$ of $v_a$ which is a strictly decreasing, concave function defined on the interval $x \in [v_a(\beta_{n,k}), a]$ which is smooth in $(v_a(\beta_{n,k}), a)$. Set $$x_{s}(a) = v_a(\beta_{n,k}).$$ We then obtain the desired $u_a : [x_{s}(a), a] \to \mathbf{R}_+$. By the computations in section \ref{prereq}, we see that $u_a$ solves equation \eqref{Main ODE y = u(x)}, and $x_{s}(a) = v_a(\beta_{n,k}) \geq \alpha_k$ can be rewritten as
$$x_{s}(a) \geq \alpha_k, \ u_a(x_{s}(a)) = \beta_{n,k}.$$
The estimates \eqref{initial condition of va cap} imply that
\begin{equation}
    u_a(a) = 0, \ \lim_{x \to a-}u_a'(x) = -\infty.
\end{equation}
Finally, rewriting the estimates \eqref{final c0 estimate for cap} and \eqref{final c1 estimate for cap} in terms of $u_a$ yield
\begin{equation}
    \beta_{n,k}\sqrt{1 - \frac{x^2 - \alpha_k^2}{a^2 - \alpha_k^2}} \leq u_a \leq \beta_{n,k},
\end{equation}
and
\begin{equation}
    u_a'(x) \leq (\frac{u_a(x)}{2} - \frac{\beta_{n,k}^2}{2u_a(x)})\frac{2x}{x^2 - \alpha_k^2} .
\end{equation}
\end{proof}

We now improve the upper bound 
$$u_a(x) \leq \beta_{n,k}$$
given in estimate \eqref{C0 estimate for caps in proposisition 4.1}. The improved estimate will be strong enough to obtain a precise asymptotics of $u_a$ on each bounded interval. The arguments in this section follow the arguments in section 8 of \cite{ADS}. \\

To obtain better estimates, we split the `disk' into its tip region and intermediate region. For each $a > \alpha_k$, define 
\begin{equation}\label{definition of rho}
    \rho = ay \in [0, \beta_{n,k}a)
\end{equation}
and write
\begin{equation}\label{definition of psi}
    x = v_a(y) = a - a^{-1}\psi(\rho, a) = a - a^{-1}\psi(ay, a).
\end{equation}
Here, $x = v_a(y)$ is the inverse function of $y = u_a(x)$ given by proposition \ref{Existence of caps}. Let $\epsilon = \frac{1}{2a^2} \in (0, \frac{1}{2\alpha_k^2})$. Then $\psi(\cdot, \epsilon) = \psi(\cdot, a)$ solves the initial value problem
\begin{equation}\label{IVP for psi}
    \begin{cases}
        \frac{\psi_{\rho\rho}(\rho, \epsilon)}{1 + (\psi_{\rho}(\rho, \epsilon))^2} = -\frac{n - k}{\rho}\psi_{\rho}(\rho, \epsilon) + \frac{1}{2} + \epsilon(\rho\psi_{\rho}(\rho, \epsilon) - \psi(\rho, \epsilon)) - \frac{2(k-1)\epsilon}{1 - 2\epsilon\psi(\rho, \epsilon)}  \\
        \psi(0, \epsilon) = \psi_{\rho}(0, \epsilon) = 0.
    \end{cases}
\end{equation}
Also, as $\epsilon \to 0$ ($a \to \infty$), the above initial value problem becomes
\begin{equation}
    \begin{cases}
        \frac{\Psi_{\rho\rho}}{1 + (\Psi_{\rho})^2} = -\frac{n - k}{\rho}\Psi_{\rho} + \frac{1}{2}  \\
        \Psi(0) = \Psi_{\rho}(0) = 0.
    \end{cases}
\end{equation}
By lemma 8.3 of \cite{ADS}, such $\Psi$ is unique and satisfies the following asymptotics as $\rho \to \infty$
\begin{equation}\label{Asymptotics of Psi}
    \begin{cases}
        \Psi(\rho) = \frac{1}{2\beta_{n,k}^2}\rho^2 - 2\ln \rho +C_0 + O(\rho^{-2}) \\
        \Psi'(\rho) = \frac{1}{\beta_{n,k}^2}\rho - \frac{2}{\rho} + O(\rho^{-3})\\
        \Psi''(\rho) = \frac{1}{\beta_{n,k}^2} + \frac{2}{\rho^2} + O(\rho^{-4}).
    \end{cases}
\end{equation}
By the change of variables
\begin{equation}
    (\rho(\tau), v(\tau), w(\tau)) = (e^{\tau}, \psi(e^{\tau}, \epsilon), \psi_{\rho}(e^{\tau}, \epsilon)),
\end{equation}
unstable manifold theory near the origin gives us the existence, uniqueness, and smooth dependence of $\psi$. 
\begin{lemma}\label{smooth dependence of cap solutions}
For each $\epsilon \in (0, \frac{1}{2\alpha_k^2})$, the initial value problem \eqref{IVP for psi} has a unique solution given by $$\psi(\rho, \epsilon) = \frac{1}{2\epsilon} - \frac{1}{\sqrt{2\epsilon}}v_{1/\sqrt{2\epsilon}}(\sqrt{2\epsilon}\rho)$$
with $\rho \in [0, \frac{\beta_{n,k}}{\sqrt{2\epsilon}})$. Here, $x = v_a(y)$ is the inverse function of $y = u_a(x)$ given by proposition \ref{Existence of caps}. Also
    $\psi(\rho, a) = \psi(\rho, \frac{1}{2a^2}) = \psi(\rho, \epsilon)$ is real analytic in $(\rho, \epsilon)$ variables with $\psi(\cdot, \epsilon) \to \Psi$ as $\epsilon \to 0$.
\end{lemma}
For each $M > 0$, $x_0 > \alpha_k$, we define \begin{equation}\label{definition of xma}
    x_{Ma} = a - a^{-1}\psi(M, a).
\end{equation}
We denote the `tip region' to be parts of the `disk' where $x \in [x_{Ma}, a]$ and the `intermediate region' to be parts of the `disk' where $x \in [\max(x_{s}(a), x_0), x_{Ma}]$. Geometrically, the tip region is where $u_a \leq \frac{M}{a}$ and the `intermediate region' is the remaining part except some parts near the boundary of the surface.\\

We now state and prove the improved estimate for `disks'.
\begin{lemma}\label{Improved C^0 estimate}
For each $2 \leq k \leq n-1$, let $u_a$ be the function given in proposition \ref{Existence of caps}. \\\\
(i) There exists $x_0 = x_0(n,k) > 0$, $M_0 = M_0(n,k)$ so that for all $M \geq M_0(n,k)$, there exists $a_0 = a_0(M, n, k) > \alpha_k$ so that one has the following estimate in the `intermediate region'.
    \begin{equation}\label{upper bound for caps in intermediate region}
        u_a(x) \leq \sqrt{\beta_{n,k}^2[1 - (1 - C(n,k)\frac{\ln a}{a^2})\frac{x^2 - c(n,k)}{a^2 - \alpha_k^2}]}
    \end{equation}
    for $x \in [\max(x_{s}(a), x_0), x_{Ma}]$, $M \geq M_0$ and $a \geq a_0$.\\\\
(ii) There exists $x_1 = x_1(n,k) > 0$ and $a_1 = a_1(n,k) > 0 $ so that 
    \begin{equation}\label{rough upper bound for x0(a)}
        x_{s}(a) \leq x_1
    \end{equation}
for all $a \geq a_1$.\\\\
 (iii) One has the following precise asymptotics of $u_a$ on each bounded interval $[x_1, N]$, namely
    \begin{equation}\label{precise asymptotics}
        u_a(x) = \beta_{n,k}(1 - \frac{x^2 - \alpha_k^2 - 2}{2(a^2 - \alpha_k^2)}) + o(a^{-2}) \text{ as }a \to \infty.
    \end{equation}
    Here, $x_1$ is the constant from (ii).
\end{lemma}
\begin{remark}
    The lower bound in estimate \eqref{C0 estimate for caps in proposisition 4.1} and the upper bound \eqref{upper bound for caps in intermediate region} is analogous to proposition 8.7 and 8.10 in \cite{ADS}. The estimate \eqref{precise asymptotics} is analogous to lemma 4.4 in \cite{ADS}.
\end{remark}
\begin{remark}\label{Upper bound of starting point for caps}
    Due to lemma \ref{Improved C^0 estimate}, by defining
  \begin{equation}
      m_0 = m_0(n,k) = \max(x_1, a_1) > \alpha_k,
  \end{equation}
    we have
\begin{equation}
    x_{s}(a) \leq m_0 \text{ for all }a \geq m_0.
\end{equation}
In particular, the `disks' extend all the way up to a fixed compact set independently of $a$.
\end{remark}
As a consequence of lemma \ref{Improved C^0 estimate}, we show that the `disks' sweep the inside of the generalized cylinder $\mathbf{R}^k \times \mathbf{S}^{n-k}(\sqrt{2(n-k)})$ outside a compact region containing the origin.
\begin{corollary}\label{sweeping out by caps}
    Let $m_0 > 0$ be the constant from remark \ref{Upper bound of starting point for caps}. For any $(x,y) \in (m_0, \infty) \times [0, \beta_{n,k})$, there exists $a \geq m_0$ so that $y = u_a(x).$
    
\end{corollary}
\begin{proof}[Proof of corollary \ref{sweeping out by caps}]
    Let $(x,y) \in (m_0, \infty) \times [0, \beta_{n,k})$. If $y = 0$, then choose $a = x \geq m_0$. If $y \in (0, \beta_{n,k})$, then by \eqref{C0 estimate for caps in proposisition 4.1}, we can choose $a_1 > a_0 = x$ so that
\begin{equation}
     0 = u_{a_0}(x)< y < \beta_{n,k}\sqrt{1 - \frac{x^2 - \alpha_k^2}{a_1^2 - \alpha_k^2}} \leq u_{a_1}(x).
\end{equation}
    By the smooth dependence of $u_a$ on its parameter $a > \alpha_k$ (lemma \ref{smooth dependence of cap solutions}), together with $x_{s}(a) \leq m_0$ by remark \ref{Upper bound of starting point for caps}, we can apply the intermediate value theorem to find $a \in (a_0, a_1)$ so that
    $$y = u_a(x).$$
\end{proof}

\begin{proof}[Proof of lemma \ref{Improved C^0 estimate}]
We define
\begin{equation}\label{def of w_a in cap}
    w = w_a(x) = (\frac{x}{2} - \frac{\alpha_k^2}{2x})\frac{d}{dx}\ln(\beta_{n,k}^2 - u_a^2(x))
\end{equation}
for $x \in (x_{s}(a), a)$. Then by equation \eqref{Main ODE y = u(x)}, $w_a$ solves the equation
\begin{equation}\label{ode for w in cap}
    (\frac{x}{2} - \frac{k-1}{x})w_a'(x) = Lw_a
\end{equation}
where
\begin{align*}
    Lw_a = (\frac{1}{2} + \frac{k-1}{x^2})w_a - (\frac{1}{2} + \frac{n-k}{u_a^2})w_a^2 + (\frac{x}{2} - \frac{k-1}{x})^2(1 + (u_a')^2)(w_a - 1).
\end{align*}
By the derivative estimate in \eqref{C1 estimate for caps in proposisition 4.1}, we have \begin{equation}
    w_{a}(x) \geq 1
\end{equation}for all $x_{s}(a) < x < a$.\\

We first obtain an estimate of $w(x_{Ma})$. 
\begin{lemma}\label{asymptotics for w in cap} For each $M > 0$
$$w(x_{Ma}) = \frac{1}{\beta_{n,k}^2}\frac{M}{\Psi'(M)} + O(a^{-2}) \text{ as }a \to \infty.$$
Also, 
$$\frac{M}{\Psi'(M)} = \beta_{n,k}^2 + \frac{2\beta_{n,k}^4}{M^2} + O(M^{-4}) \text{ as }M \to \infty.$$
\end{lemma}
\begin{proof}
  By \eqref{definition of xma} and \eqref{Asymptotics of Psi}, we have
  \begin{equation}
      x_{Ma} = a - \frac{1}{a}\psi(M,a) = a + O(a^{-1}) \text{ as }a \to \infty,
  \end{equation}
  and
  \begin{equation}
      u_a(x_{Ma}) = \frac{M}{a}, \ u_a'(x_{Ma}) = -\psi_{\rho}(M,a)^{-1}.
  \end{equation}
  Thus, by the definition of $w = w_a$ given by \eqref{def of w_a in cap}, we see that
  \begin{align*}
      w(x_{Ma}) & = (\frac{x_{Ma}^2 - \alpha_k^2}{x_{Ma}})(\frac{M/a}{\beta^2_{n,k} - (M/a)^2})\frac{1}{\psi_{\rho}(M,a)} \\ & = (\frac{M}{\beta^2_{n,k}}+ O(a^{-2}))\frac{1}{\psi_{\rho}(M,a)}.
  \end{align*}
  By using lemma \eqref{smooth dependence of cap solutions} together with \eqref{Asymptotics of Psi}, we see that
  \begin{equation}
      \psi_{\rho}(M,a) = \Psi'(M) + O(a^{-2}) \text{ as }a \to \infty,
  \end{equation}
  which gives us the first desired inequality. \\

  The second inequality is immediate from the asymptotic formula of $\Psi$ given by \eqref{Asymptotics of Psi}.
\end{proof}
We will now obtain an upper bound of $w$ in the intermediate region.
\begin{lemma}\label{derivative estimate for caps}[cf proposition 8.9 in \cite{ADS}]
    There exists $x_0 = x_0(n,k) > 0$, $K = K(n,k) > 0$ $M_0 = M_0(n,k) > 0$ and $a_0 = a_0(M, n, k) > 0$ so that 
    $$1 \leq w(x) \leq 1 + \frac{K}{x^2} + \frac{K}{a^2 - x^2}$$
    for all $x \in (\max(x_{s}(a), x_0), x_{Ma}]$, $M \geq M_0$, $a \geq a_0$.
\end{lemma}
\begin{proof}[Proof of lemma \ref{derivative estimate for caps}]
Choose
\begin{equation}\label{choice of K and x0}
    K = 100,  \ x_0 = 10 \alpha_k
\end{equation}
For each $a > \alpha_k$, define 
\begin{equation}
    \overline{w}(x) = 1 + \frac{K}{x^2} + \frac{K}{a^2 - x^2} = 1 + \frac{Ka^2}{x^2(a^2 - x^2)}.
\end{equation}
    By the lower bound 
    $$\beta_{n,k}\sqrt{1 - \frac{x^2 - \alpha_k^2}{a^2 - \alpha_k^2}} \leq u_a,$$
     we have
    \begin{equation}\label{estimate of coefficient}
        \frac{1}{2} + \frac{n-k}{u_a^2} \leq \frac{a^2}{a^2 - x^2} = \frac{x^2}{K}(\overline{w}(x)-1).
    \end{equation}
    Also, by \eqref{Asymptotics of Psi}, we have
    $$a^2 - x_{Ma}^2 = 2\psi(M, a) + O(a^{-2}) = 2\Psi(M) + O(a^{-2}).$$
    Then
    $$\max_{x \in (\max(x_{s}(a), x_0), x_{Ma}]}\overline{w}(x) \leq 1 + \max(\frac{2K}{x_0^2}, \frac{2K}{\Psi(M) + O(a^{-2})}).$$
    By fixing $M$ large so that $\Psi(M) > 4K$ (which is possible due to \eqref{Asymptotics of Psi}), and then choosing $a_0 = a_0(M, K, n, k)$ so that $$O(a^{-2}) < 2K$$
    for all $a > a_0$, we have 
    \begin{equation}\label{estimates of overline w inbetween 1 and 2}
        1 < \overline{w}(x) <2 \text{ for }x \in (\max(x_{s}(a), x_0), x_{Ma}].
    \end{equation}
    We first show that
\begin{equation}\label{inequality for overline w in cap}
    (\frac{x}{2} - \frac{k-1}{x})\overline{w}'(x) \leq L\overline{w},
\end{equation}
where $L$ is given by \eqref{ode for w in cap}
for all $x \in (\max(x_{s}(a), x_0), x_{Ma}]$. By \eqref{estimate of coefficient}, \eqref{estimates of overline w inbetween 1 and 2}, and the choice of $K$ and $x_0$ given in \eqref{choice of K and x0}, we have
\begin{align*}
    L\overline{w} &\geq (\frac{1}{2} + \frac{k-1}{x^2})\overline{w}(x) - (\frac{1}{2} + \frac{n-k}{u_a^2})\overline{w}^2 + (\frac{x}{2} - \frac{k-1}{x})^2(\overline{w}-1)\\ &\geq -\frac{4x^2}{K}(\overline{w}-1) + (\frac{x}{2} - \frac{k-1}{x})^2(\overline{w} - 1) \\ & \geq (\frac{x^2}{9} - \frac{4x^2}{K})(\overline{w} - 1)
\end{align*}
for $x \in (\max(x_{s}(a), x_0), x_{Ma}]$. On the other hand
\begin{align*}
     (\frac{x}{2} - \frac{k-1}{x})\overline{w}'(x) & \leq \frac{Kx^2}{(a^2 - x^2)^2} \\ & \leq \frac{x^2(\overline{w}-1)}{a^2 - x_{Ma}^2}\\&  = \frac{x^2(\overline{w}-1)}{2\Psi(M)+ O(a^{-2})} \\ & \leq \frac{x^2(\overline{w}-1)}{6K }
\end{align*}
    by our previous choices of $M$ and $a_0$. Since $K = 100$, we see that 
 \eqref{inequality for overline w in cap} is indeed true
    for $x \in (\max(x_{s}(a), x_0), x_{Ma}]$, $a \geq a_0$.\\

    We now show that $$w(x_{Ma}) < \overline{w}(x_{Ma})$$
    for large enough $M$ and $a$. By lemma \ref{asymptotics for w in cap}, we have 
    $$w(x_{Ma}) = \frac{1}{\beta_{n,k}^2}\frac{M}{\Psi'(M)} + O(a^{-2})$$
    with
    $$\frac{1}{\beta_{n,k}^2}\frac{M}{\Psi'(M)} = 1 + \frac{2\beta_{n,k}^2}{M^2} + O(M^{-4}).$$
    On the other hand
    $$\overline{w}(x_{Ma}) > 1 + \frac{K}{a^2 - x_{Ma}^2} = 1 + \frac{K}{2\Psi(M) + O(a^{-2})}.$$
    The choice of $K$ given in \eqref{choice of K and x0} and the asymptotics of $\Psi$ given in \eqref{Asymptotics of Psi} implies that by fixing large enough $M \geq M_0(n,k)$ and then choosing $a_0 = a_0(M, K,n,k)$, one has
    \begin{equation}\label{comparison of w and overline w at tip in cap case}
        w(x_{Ma}) < \overline{w}(x_{Ma})
    \end{equation}
    for all $a \geq a_0$.
    Thus by comparison principle, we have
    $$w(x) \leq \overline{w}(x) \text{ for }x \in (\max(x_{s}(a), x_0), x_{Ma}]$$
    for all $M \geq M_0(n,k)$ and $a \geq a_0(M, n,k)$.
\end{proof}
To establish (i) of lemma \ref{Improved C^0 estimate}, we first integrate the upper bound in lemma \ref{derivative estimate for caps}. This gives us the estimate 
\begin{equation}
    \ln \frac{\beta_{n,k}^2 -(M/a)^2}{\beta_{n,k}^2 - u_a^2(x)}  \leq \int_{x}^{x_{Ma}}\frac{2x}{x^2 -\alpha_k^2}(1 + \frac{K}{x^2} + \frac{K}{a^2 - x^2})dx.
\end{equation}
By using $e^{-x} \geq 1 - x$, we obtain
\begin{equation}
    u_a^2(x) \leq \beta_{n,k}^2[1 -(1 - C(n,k)\frac{\ln a}{a^2})\frac{x^2 - c(n,k)}{a^2 - \alpha_k^2}]
\end{equation}
for $x \in (\max(x_{s}(a), x_0), x_{Ma}]$ for all $M \geq M_0(n,k)$ and $a \geq a_0(M, n,k)$. By continuity of $u_a$, above estimate holds up to $x = \max(x_{s}(a), x_0)$ thus proving (i) of lemma \ref{Improved C^0 estimate}.\\

 To prove (ii) of lemma \ref{Improved C^0 estimate}, we first choose $x_1 = \max(2c(n,k), x_0(n,k))$ where $x_0$ and $c(n,k)$ are from the estimate \eqref{upper bound for caps in intermediate region}. Fix $M = M_0(n,k)$ and then fix $a_1 = a_1(n,k) \geq a_0(M_0, n,k)$ so that
  $$x_{Ma} > x_1, \ C(n,k)\frac{\ln a}{a^2} < \frac{1}{2}, \ \frac{M_0}{a} \leq \frac{\beta_{n,k}}{2}\text{ for all }a \geq a_1.$$
Here, $C(n,k)$ is a constant in \eqref{upper bound for caps in intermediate region}.
Then for each $a \geq a_1$, we see that $x_{s}(a) \leq x_{Ma}$ because 
$$u_a(x_{Ma}) = \frac{M_0}{a} \leq \frac{\beta_{n,k}}{2} < \beta_{n,k} = u_a(x_{s}(a)).$$
Also, the upper bound in \eqref{upper bound for caps in intermediate region} is strictly smaller than $\beta_{n,k}$ for all $x \in [x_1, x_{Ma}]$, $a \geq a_1$. Because $u_a(x_{s}(a)) = \beta_{n,k}$, we see that $x_{s}(a) \leq x_1 $ for all $a \geq a_1$, thus proving (ii) of lemma \ref{Improved C^0 estimate}.\\

We now prove the asymptotics \eqref{precise asymptotics}. Since this asymptotics is not used in other parts of this paper, and the argument is essentially the same as the proof of lemma 4.4 in \cite{ADS}, we only give a brief sketch of the proof. By combining \eqref{upper bound for caps in intermediate region} and the lower bound of $u_a$ in \eqref{C0 estimate for caps in proposisition 4.1} in proposition \ref{Existence of caps}, and following the proof of proposition 8.11 in \cite{ADS}, we obtain an estimate
    \begin{equation}\label{intermediate estimate for asymptotics of caps}
        |u_a - \beta_{n,k}(1 - \frac{x^2 - \alpha_k^2}{2(a^2 - \alpha_k^2)})| \leq \frac{C_{n,k}}{a^2}
    \end{equation}
    for all $x \in [x_1, 4L]$ for all sufficiently large $a$. Define $v_a$ by the following relation
    \begin{equation}
        u_a = \beta_{n,k}(1 + \frac{v_a}{a^2 - \alpha_k^2}).
    \end{equation}
    From the estimates \eqref{intermediate estimate for asymptotics of caps}, we obtain 
    \begin{equation}
        |v_a -(-\frac{x^2 - \alpha_k^2 -2}{2})| \leq C_{n,k}
    \end{equation}
    for $x \in [x_1, 4L]$ for all $L \geq x_1$ and all sufficiently large $a$. This estimate together with equation \eqref{Main ODE y = u(x)}, and the fact that $v_a$ is concave, we can extract a $C^2_{loc}([x_1, \infty))$ subsequential limit $v_a \to v$ as $a \to \infty$.  This $v$ solves the equation
    $$v_{xx} = (\frac{x}{2} - \frac{k-1}{x})v_x - v.$$
    By finding the general solutions to above equation, we can write
    $$v = \alpha(x^2 - \alpha_k^2 - 2) + \beta v_1$$
    for some constant $\alpha, \beta \in \mathbf{R}$ 
    where $v_1$ is another general solution to above equation given by
 \begin{equation}
     v_1(x) = (x^2 - \alpha_k^2-2)\int_{x_1}^{x}\frac{e^{\eta^2/4}}{\eta^{k-1}(\eta^2 - \alpha_k^2 - 2)^2}d\eta
 \end{equation}
    which has an exponential growth as $x \to \infty$.
    This $v_1$ can be found by writing
    $$v_1(x) = \psi(x)(x^2 - \alpha_k^2 - 2)$$
    and solving for $\psi$.
    By following the arguments of the proof of lemma 4.4 in page 444 to 445 in \cite{ADS}, we obtain the precise asymptotics \eqref{precise asymptotics}.
\end{proof}

\section{Proof of theorem 
\ref{Main theorem : Foliation}}\label{section 5}
In this section, we prove theorem \ref{Main theorem : Foliation}. To more precisely state the main result, we first define the reflection map. 
\begin{definition}\label{reflection map}
    Let $2 \leq k \leq n-1$. Then the reflection map $R_{n,k}$ is defined by
    \begin{equation}
        R_{n,k} : \mathbf{R}^{n-k+1}\times \mathbf{R}^{k} \ni (y,x) \to (x,y) \in \mathbf{R}^{n+1}.
    \end{equation}
\end{definition}
We now state the main result of this section.
\begin{theorem}\label{specific form of theorem 1.3}
    Let $2 \leq k \leq n-1$. Then there exists $R_0 = R_0(n,k) > 0$ so that $\mathbf{R}^{n+1} \setminus B(0, R_0)$ is foliated by the following $O(k)\times O(n-k+1)$ symmetric self-shrinkers. 
    \begin{itemize}
        \item $\Sigma^{n,k}_{\sigma}$ with $\sigma \in (0, \sigma_{n,k})$ (`trumpets' constructed in proposition \ref{Existence of trumpets}).
        \item $R_{n,k}(\Sigma^{n, n-k+1}_{\sigma})$ with $\sigma \in (0, \sigma_{n, n-k+1})$ (reflected `trumpets').
        \item $\Gamma^{n,k}_a$ (`disks' constructed in proposition \ref{Existence of caps}).
        \item $R_{n,k}(\Gamma^{n,n-k+1}_a)$  (reflected `disks'). 
        \item $C_{n,k} = \mathbf{S}^{k-1}(\alpha_k) \times \mathbf{R}^{n-k+1}$ ($O(k)\times O(n-k+1)$ symmetric cylinder).
        \item $\Tilde{C}_{n,k} = \mathbf{R}^k\times \mathbf{S}^{n-k}(\beta_{n,k})$ ($O(k)\times O(n-k+1)$ symmetric cylinder).
        \item $\Lambda_{n,k} = \{(x,y) \in \mathbf{R}^k\times \mathbf{R}^{n-k+1}\ | \ |y| = \sigma_{n,k}|x| \}$ ($O(k)\times O(n-k+1)$ symmetric minimal cone).
    \end{itemize}
\end{theorem}
By using the reduction onto the first quadrant given in section \ref{prereq}, we can visualize the surfaces considered in theorem \ref{specific form of theorem 1.3}. For example, when we take $(n,k) = (3,2)$, then we get the following picture.
\begin{figure}[!h]
    \centering
    \centerline{\includegraphics[width=7cm]{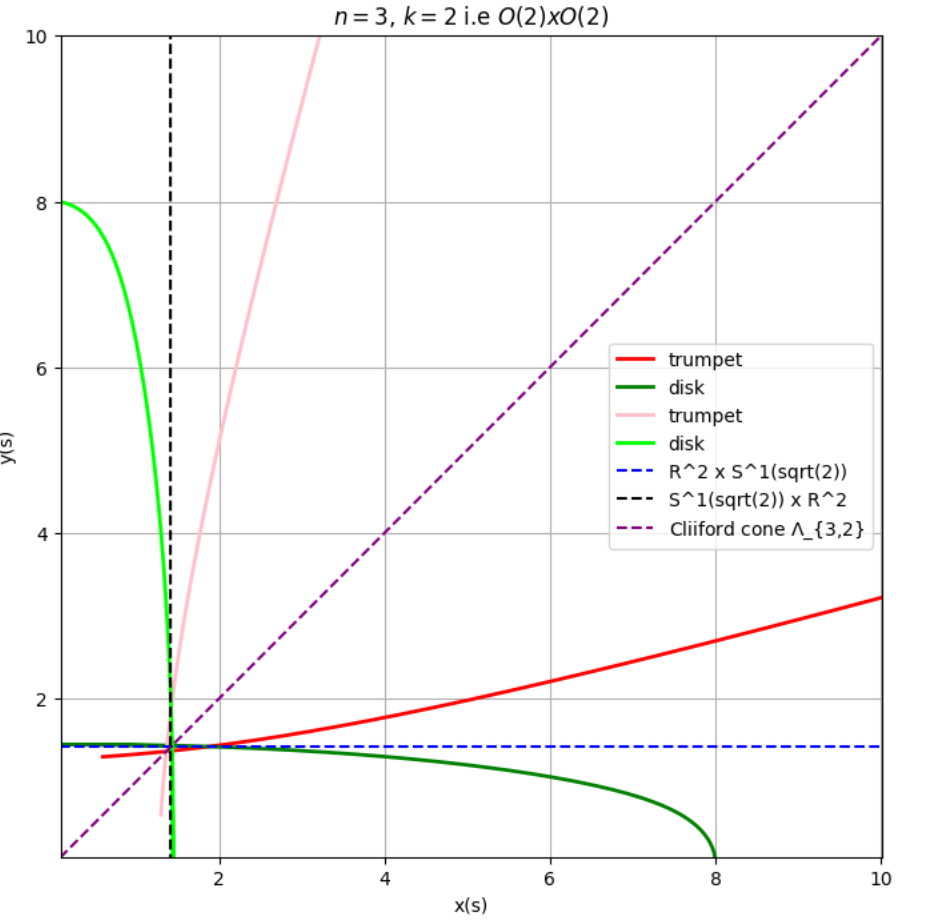}}
    \label{fig:label}
    \caption{Some examples of `trumpets' and `disks' with $O(2) \times O(2)$ symmetry}
\end{figure}

We now begin the proof of theorem \ref{specific form of theorem 1.3}. We first show that the `disks' foliate the inside of generalized cylinders except some compact set containing the origin. Since we already showed that the `disks' sweep out such region (corollary \ref{sweeping out by caps}), it is enough to show that these `disks' vary `monotonically' in $a$. Since we are going to remove a compact set near the origin, we focus on parts of `disks' where $x \geq m_0$ with $m_0$ being the constant from remark \ref{Upper bound of starting point for caps}. Then for each $a \geq m_0$, the `disks' are parametrized by
$$X_a(x, w_1, w_2) : [m_0, a] \times \mathbf{S}^{k-1} \times \mathbf{S}^{n-k} \to \mathbf{R}^{n+1}, \ \ X_a(x,w_1, w_2) = (xw_1, u_a(x)w_2).$$
One can now define the normal variation of $X_a$ given by
\begin{equation}\label{normal variation of caps}
  V = V_a(x) = \nu \cdot \frac{\partial X_a}{\partial a} = \frac{1}{\sqrt{1 + (u_a')^2}}\frac{\partial u_a}{\partial a}.  
\end{equation}
We will prove that the `disks' form a foliation by showing that for sufficiently large $x$ and $a$, the normal variation \eqref{normal variation of caps} is positive. This will be done by comparing $V$ with a suitable barrier function which is a strict supersolution to the linearization of the self shrinker equation. 
\begin{lemma}\label{injectivity of caps}
    There exists $x_0 = x_0(n,k) > 0$ and $a_0 = a_0(n,k) > 0$ so that $V > 0$ on parts of `disks' where $x \geq x_0$ for $a \geq a_0$.
\end{lemma}
\begin{proof}[Proof of lemma \ref{injectivity of caps}]
    The proof is essentially same as the proof of lemma 8.12 in \cite{ADS}. We define the linearization of the self shrinker equation
    \begin{equation}\label{linearized operator}
        \mathcal{L}_a(f)  = \Delta f - \nabla \phi \cdot \nabla f + (|A|^2 + \frac{1}{2})f
    \end{equation}
    where $A$ is the second fundamental form of the `disk' $\Gamma^{n,k}_a$, and $\phi = \frac{|X|^2}{4}$. We define the following barrier function
    \begin{equation}\label{Barrier function W}
        W = e^{\phi/2} = e^{|X|^2/8}.
    \end{equation}
    We first look at the tip region. 
\begin{lemma}\label{Asymptotic expansion of V in cap}Let $V$ be the normal variation of `disks' given in \eqref{normal variation of caps}. Then the following holds. \\

(i) For all $a > \alpha_k$, 
    $$\mathcal{L}(V) = 0$$
     on the `disk' in theorem \ref{Main theorem : existence of birotational caps} generated by $u_a$ from proposition \ref{Existence of caps}.\\
     
    (ii) For each $M > 0$,
    $$V = \frac{1}{\sqrt{1 + (\Psi')^2}} + O(a^{-2}) \text{ as }a \to \infty$$
    uniformly in $\rho \in [0, M]$ where $\rho$ is given by \eqref{definition of rho}. Also
    $$\frac{V_{\rho}}{V}(x_{Ma}) = - \frac{\Psi'(M)\Psi''(M)}{1 + \Psi'(M)^2} + O(a^{-2}) \text{ as }a \to \infty.$$
    When $\rho = 0$ i.e at the tip, $V = 1$. 
\end{lemma}
\begin{proof}
    Item (i) is clear since each $\Gamma^{n,k}_a$ is a mean curvature flow self shrinker. \\

    To prove (ii), we recall that the tip region ($\rho = ay \in [0, M]$) is parametrized by
    \begin{equation}
        [0,M] \times \mathbf{S}^{k-1}\times \mathbf{S}^{n-k} \ni (\rho, w_1, w_2) \to ((a - a^{-1}\psi(\rho, a))w_1, \frac{\rho}{a}w_2),
    \end{equation}
where $\psi$ is defined in \eqref{definition of psi}. Then we have 
\begin{equation}
    u_a(a - a^{-1}\psi(\rho, a)) = \frac{\rho}{a}.
\end{equation}
Differentiating in $\rho$, and $a$ respectively, we obtain
\begin{equation*}
     u_a'(a - a^{-1}\psi(\rho, a)) = -\psi_{\rho}(\rho, a)^{-1}, 
\end{equation*}
and
\begin{equation*}
    \partial_au_a(a - a^{-1}\psi(\rho, a)) = -\frac{\rho}{a^2} +\psi_{\rho}(\rho, a)^{-1}(1 + \frac{1}{a^2}\psi(\rho, a) - \frac{1}{a}\partial_a\psi(\rho, a)).
\end{equation*}
Then item (ii) follows by using lemma \ref{smooth dependence of cap solutions} to replace $\psi(\rho, a)$ with $\Psi(\rho)$. 
\end{proof}
Lemma \ref{Asymptotic expansion of V in cap} immediately implies following corollary which says that the normal variation is strictly positive in the `tip region'. 
\begin{corollary}\label{nonvanishing of normal variation at tip}
    For each $M > 0$, one can find $\hat{a}_0 = \hat{a}_0(M, n,k) > 0$ so that $$V > 0$$
for all $x \geq x_{Ma}, a \geq \hat{a}_0$. 
\end{corollary}
To control the `intermediate region', we 
need the following lemma for $W$.
\begin{lemma}\label{properties of W on disk}
    Let $W$ be as in \eqref{Barrier function W}. There exists $\Tilde{x}_0 = \Tilde{x}_0(n,k) > 0$, $\Tilde{M}_0 = \Tilde{M}_0(n,k) > 0$ and $\Tilde{a}_0 = \Tilde{a}_0(M, n,k) > 0$ so that
    $$\mathcal{L}(W) < 0$$
    on parts of `disk' where $\Tilde{x}_0 \leq x \leq x_{Ma}$ for each $M \geq \Tilde{M}_0$ and $a \geq \Tilde{a}_0$.\\
    
Also, for each $M > 0$
$$\frac{W_{\rho}}{W}(x_{Ma}) = -\frac{1}{4}\Psi'(M) + O(a^{-2})$$
as $a \to \infty$.
\end{lemma}
\begin{proof}[Proof of lemma \ref{properties of W on disk}]
    By direct calculation, we see that
\begin{equation}\label{estimation of L(W) for caps}
   \frac{ \mathcal{L}(W)}{W} \leq \frac{n + 2}{4} + |A|^2 - \frac{1}{16}|X|^2.
\end{equation}
By the computations given in section \ref{prereq}, 
\begin{equation}
    |A|^2 \leq (\frac{-u_{a}''}{1 + (u_{a}')^2})^2 + (k-1)(\frac{u_{a}'}{x})^2 + (n-k)\frac{1}{u_{a}^2}.
\end{equation}
By using concavity of $u_a$, together with the equation \eqref{Main ODE y = u(x)}, and the fact that $u_a' \leq 0, u_a > 0$, we have
\begin{equation}
    0 < -\frac{u_a''}{1 + (u_a')^2} \leq -\frac{xu_a'}{2} + \frac{u_a}{2}.
\end{equation}
Therefore, we obtain
\begin{align*}
    |A|^2 &\leq (\frac{-u_{a}''}{1 + (u_{a}')^2})^2 + (k-1)(\frac{u_{a}'}{x})^2 + (n-k)\frac{1}{u_{a}^2}\\& \leq \frac{x^2(u_a')^2}{2} + \frac{u_a^2}{2} + (k-1)(\frac{u_{a}'}{x})^2 + (n-k)\frac{1}{u_{a}^2} \\ & \leq x^2(u_a')^2 + \frac{n-k}{u_a^2} + C(n,k).
\end{align*}
Here, we used the fact that $0 \leq u_a \leq \beta_{n,k}$, and that $x \geq 
\Tilde{x}_0(n,k)$. \\

Since $u_a' < 0, \ u_a'' \leq 0$, and $x \leq x_{Ma}$, we have
\begin{equation}
    (u_a'(x))^2 \leq (u_a'(x_{Ma}))^2 = \frac{1}{\psi_{\rho}(M, a)^2}.
\end{equation}
Also, since $\frac{1}{u_a} > 0$ is convex, by interpolation between $\Tilde{x_0}(n,k)$ and $x_{Ma}$, we have
\begin{equation}
    \frac{1}{u_a^2} \leq c(n,k)(\frac{x^2}{M^2} + 1) 
\end{equation}
for every $a \geq \Tilde{a}_0(M,n,k)$ for each fixed large $M > 0$. Therefore, we obtain
\begin{equation}\label{estimate of |A| for caps}
    |A|^2 \leq c(n,k)(\frac{1}{\Psi'(M)^2} + \frac{1}{M^2} + O(a^{-2}))x^2 + C(n,k).
\end{equation}
Taking into account the asymptotics of $\Psi$ given in \eqref{Asymptotics of Psi}, one can choose $\Tilde{M}_0 = \Tilde{M}_0(n,k) > 0$ so that for each $M \geq \Tilde{M}_0$, one can choose $\Tilde{a}_0 = \Tilde{a}_0(M, n, k) > 0$ so that 
\begin{equation}\label{choices of tildeM0 and tildea0}
    c(n,k)(\frac{1}{\Psi'(M)^2} + \frac{1}{M^2} + O(a^{-2})) < \frac{1}{32}
\end{equation}
for all $a \geq \Tilde{a}_0$. Then by \eqref{estimation of L(W) for caps}, \eqref{estimate of |A| for caps}, and \eqref{choices of tildeM0 and tildea0}, we see that there exists $\Tilde{x}_0 = \Tilde{x}_0(n,k) > 0$ so that whenever $x \geq \Tilde{x}_0$, $M \geq \Tilde{M}_0$ and $a \geq \Tilde{a}_0$, we obtain
\begin{equation}\label{LW is supersolution}
    \frac{\mathcal{L}(W)}{W} \leq \frac{n + 2}{4} + |A|^2 - \frac{1}{16}|X|^2  \leq C(n,k) - \frac{x^2}{32} < 0
\end{equation}
thus proving the first part of lemma \ref{properties of W on disk}. The proof of the second formula is an immediate consequence of the definition of $x_{Ma}$ given by
\begin{equation}
    u_a(x_{Ma}) = u_a(a - a^{-1}\psi(M, a)) = \frac{M}{a}
\end{equation}
together with lemma \ref{smooth dependence of cap solutions}.
\end{proof}
Once lemma \ref{properties of W on disk} is established, one can proceed as in section 8.7.5 of \cite{ADS}. Choose $x_0 = \Tilde{x}_0 > 0$ with $\Tilde{x}_0$ being the constant from lemma \ref{properties of W on disk}. Assume there exists $\overline{x} \geq x_0$ so that $V(\overline{x}) \leq 0$. We then consider $$f(x) = \frac{V}{W}$$
on $[x_0, a]$. Let us fix $M > 0$ for the moment. By corollary \ref{nonvanishing of normal variation at tip}, there exists $\hat{a}_0(M, n,k) > 0$ so that $f > 0$ on $[x_{Ma}, a]$ for all $a \geq \hat{a}_0$. This means $\overline{x} < x_{Ma}$, hence by possibly choosing the maximal $\overline{x}$, without loss of generality, we may assume that $f(\overline{x}) = 0$ and $f > 0$ for all $x > \overline{x}$. Then there exists $p \in (\overline{x}, x_{Ma}]$ so that
$$f(p) = \max_{[\overline{x}, x_{Ma}]}f(x) > 0.$$
If $p < x_{Ma}$, then by maximum principle together with $f(p) > 0$, and (i) of lemma \ref{Asymptotic expansion of V in cap}
$$\mathcal{L}(W) \geq 0.$$
This contradicts lemma \ref{properties of W on disk} if $M \geq \Tilde{M}_0$ and $a \geq \Tilde{a}_0(M, n, k)$. Thus $p = x_{Ma}$. Combining with the fact that $\rho$ decreases as $x$ increases, we obtain 
$$\frac{d}{d\rho}f(x_{Ma}) \leq 0,$$
hence
$$\frac{V_{\rho}}{V}(x_{Ma}) \leq \frac{W_{\rho}}{W}(x_{Ma}).$$
By (ii) of lemma \ref{Asymptotic expansion of V in cap}, lemma \ref{properties of W on disk}, and the asymptotics of $\Psi$ given in \eqref{Asymptotics of Psi}, we have
\begin{equation}\label{comparsion of V and W at boundary of tip and intermideiate}
    - \frac{\Psi'(M)\Psi''(M)}{1 + \Psi'(M)^2} + O(a^{-2}) \leq -\frac{1}{4}\Psi'(M) + O(a^{-2})
\end{equation}
with 
$$- \frac{\Psi'(M)\Psi''(M)}{1 + \Psi'(M)^2} = -\frac{1}{M} + O(M^{-3}),$$
and
$$-\frac{1}{4}\Psi'(M) = -\frac{M}{4\beta_{n,k}^2} + O(M^{-1}).$$
We can first choose $M_0 = M_0(n,k) \geq \Tilde{M}_0$, and then 
\begin{equation}
    a_0 = a_0(n,k) \geq \max(\hat{a}(M_0, n,k), \Tilde{a}_0(M_0,n,k))
\end{equation}
so that corollary \ref{nonvanishing of normal variation at tip}, lemma \ref{properties of W on disk} hold with $M = M_0$, and
$$-\frac{\Psi'(M_0)\Psi''(M_0)}{1 + \Psi'(M_0)^2} + O(a^{-2}) > -\frac{1}{4}\Psi'(M_0) + O(a^{-2})$$
for all $a \geq a_0$. Such choices of $M_0$ and $a_0$ gives a contradiction to  \eqref{comparsion of V and W at boundary of tip and intermideiate}. This implies that $V > 0$ for all $x \geq x_0$ and $a \geq a_0$, thus completing the proof of lemma \ref{injectivity of caps}. 
\end{proof}

As a consequence of lemma \ref{injectivity of caps}, we obtain a foliation of the inside of generalized cylinder by `disks'.
\begin{proposition}\label{Foliation by lower caps}
    There exists $x_1 = x_1(n,k) > 0$ so that the region $$R = \{(xw_1, yw_2) \ | \ x \geq x_1, 0 \leq y < \beta_{n,k} , w_1 \in \mathbf{S}^{k-1}, w_2 \in \mathbf{S}^{n-k}\}$$ is foliated by parts of `disks' which lie in $R$. 
\end{proposition}
\begin{proof}[Proof of proposition \ref{Foliation by lower caps}]
    We take $x_1 = \max(x_0, a_0, m_0) > 0$, where $x_0$ and $a_0$ are constants from lemma \ref{injectivity of caps}, and $m_0$ is a constant from remark \ref{Upper bound of starting point for caps}. Then by lemma \ref{injectivity of caps} and definition of $V$ given in \eqref{normal variation of caps}, $a \to u_{a}(x)$ is monotonically increasing for all $x \geq x_1$ and $a \geq x_1$. Combining this with corollary \ref{sweeping out by caps}, we see that entire $R$ is covered by `disks' for $a \geq x_1$. This completes the proof.  
\end{proof}

We now show that the `trumpets' foliate region between a minimal cone and the cylinder except some compact region containing the origin. Once again, the idea is to show that the normal variation is positive on parts of the `trumpets' sufficiently away from the origin. Since we are focusing on far-field, we focus on parts of `trumpets' where $x \geq l_0$ with $l_0$ being the constant from corollary \ref{Sweeping out by trumpets}. Let us recall the definition of $\alpha_k$, $\beta_{n,k}$ and $\sigma_{n,k}$ given in \eqref{frequently used constants}. \\

We now discuss foliation by `trumpets'. The `trumpets' can be parametrized by
$$ X_{\sigma}(x, w_1, w_2) : [l_0, \infty) \times \mathbf{S}^{k-1} \times \mathbf{S}^{n-k} \to \mathbf{R}^{n+1}, \ X_{\sigma}(x, w_1, w_2) = (xw_1, u_{\sigma}(x)w_2)$$
where $u_{\sigma}$ is given by proposition \ref{Existence of trumpets}.
We then can define the normal variation of `trumpets' by
\begin{equation}\label{normal variation of trumpet}
    V = \nu \cdot \frac{\partial X_{\sigma}}{\partial \sigma} = \frac{1}{\sqrt{1 + (u_{\sigma}')^2}}\frac{\partial u_{\sigma}}{\partial \sigma}.
\end{equation}
We prove the positivity of $V$ on parts of `trumpet' sufficiently away from the origin by using the barrier function given in \eqref{Barrier function W}.
\begin{lemma}\label{nonvanishing of normal variation of trumpets}
    There exists $x_0 = x_0(n,k) > l_0$ so that the normal variation $V$ given in \eqref{normal variation of trumpet} is positive 
    on parts of `trumpet' where $x \geq x_0$.
\end{lemma}
\begin{proof}[Proof of lemma \ref{nonvanishing of normal variation of trumpets}]
     Fix $\sigma \in (0, \sigma_{n,k})$. 
    By the asymptotics of $u_{\sigma}$ as $x \to \infty$ given in proposition \ref{Existence of trumpets}, we have
    \begin{equation}\label{asymptotic expansion of V as x to infty}
        V(x) = \frac{x + O(x^{-1})}{\sqrt{1 + (u_{\sigma}')^2}}
    \end{equation}
    as $x \to \infty$. By the definition of $\mathcal{L}$ given in \eqref{linearized operator} and the fact that each $\Sigma^{n,k}_{\sigma}$ is a self shrinker, we have
    \begin{equation}
        \mathcal{L}(V) = 0.
    \end{equation}
    Let $W$ be the function given by \eqref{Barrier function W}. We claim that 
    \begin{equation}\label{W is supersolution on trumpets}
        \mathcal{L}(W) < 0
    \end{equation}
    for $x \geq x_0(n,k)$. The proof is essentially the same as that of lemma \ref{injectivity of caps}. Just like in the `disk' case \eqref{estimation of L(W) for caps}, we once again have
    \begin{equation}\label{L(W) estimate for trumpets}
        \frac{\mathcal{L}(W)}{W} \leq \frac{n+2}{4} + |A|^2 - \frac{1}{16}|X|^2.
    \end{equation}
    By proposition \ref{Existence of trumpets}, $u_{\sigma} \geq \beta_{n,k}$, and $u_{\sigma}'' \geq 0$. Also, by lemma \ref{Derivative bound for trumpets}, we have
    $$1 \leq w_{\sigma}(x) \leq  1 + \frac{72}{x^2} $$
    for all $\sigma \in (0, \sigma_{n,k})$ and $x \geq l_0$. Hence we can see that \begin{equation}
        0 \leq \frac{u_{\sigma}''}{1 + (u_{\sigma}')^2} \leq (\frac{u_{\sigma}}{2} - \frac{n-k}{u_{\sigma}})(w_{\sigma}-1) \leq \frac{\sigma_{n,k}x}{2}( \frac{72}{x^2}) \leq C(n,k).
    \end{equation} 
     Therefore by the computations in section \ref{prereq}, we have
     \begin{equation}\label{estimation of |A| for trumpet}
       |A|^2 \leq (\frac{-u_{\sigma}''}{1 + (u_{\sigma}')^2})^2 + (k-1)(\frac{u_{\sigma}'}{x})^2 + (n-k)\frac{1}{u_{\sigma}^2} \leq C(n,k). 
    \end{equation}
   By choosing $x_0 = x_0(n,k) > l_0$ so that
    \begin{equation}\label{choice of x5}
        C(n,k) + \frac{n + 2}{4} - \frac{1}{20}x_0^2 < 0,
    \end{equation}
    we have by \eqref{L(W) estimate for trumpets}, \eqref{estimation of |A| for trumpet}, and the choice of $x_0$ given by \eqref{choice of x5}, the following estimate
    \begin{equation}
        \frac{\mathcal{L}(W)}{W}\leq \frac{n+2}{4} + C(n,k) - \frac{1}{16}|X|^2 \leq \frac{n+2}{4} + C(n,k) - \frac{1}{16}x^2 < 0
    \end{equation}
    for all $x \geq x_0$, $\sigma \in (0, \sigma_{n,k})$.\\
    
    Now one consider $\frac{V}{W}$. It is clear by \eqref{asymptotic expansion of V as x to infty} and the definition of $W$ given in \eqref{Barrier function W} that $$\frac{V}{W}\to 0$$ as $x \to \infty$ and is positive for large $x$. If $V$ vanishes at some $x \geq x_0$, then one can find strictly positive interior maximum of $V/W$. Because $\mathcal{L}(V) = 0$, by the second derivative test, one has
    $$\mathcal{L}(W) \geq 0$$
    at the point where the maximum of $\frac{V}{W}$ is attained. This contradicts inequality \eqref{W is supersolution on trumpets}, hence $V$ does not vanish for $x \geq x_0$. Because $V > 0$ for large $x$, this means $V > 0$ for all $x \geq x_0$, thus proving lemma \ref{nonvanishing of normal variation of trumpets}.
\end{proof}
Combining with the estimates in lemma \ref{Alternative $C^0$ for trumpets}, we obtain the following proposition.
\begin{proposition}\label{FOliation by lower trumpets}
    There exists $x_0 = x_0(n,k) > 0$ so that the region
    $$R = \{(xw_1, yw_2) \ | \ x \geq x_0, \beta_{n,k} < y < \sigma_{n,k}x , w_1 \in \mathbf{S}^{k-1}, w_2 \in \mathbf{S}^{n-k}\}$$
    is foliated by parts of the `trumpets' which lie in $R$.
\end{proposition}
\begin{proof}[Proof of proposition \ref{FOliation by lower trumpets}]
Choose $x_0 > l_0$ from lemma \ref{nonvanishing of normal variation of trumpets}. Then we see that for each $x \geq x_0(n,k)$ 
$$\sigma \to u_{\sigma}(x)$$
is monotonically increasing. Also, lemma \ref{Sweeping out by trumpets} implies that the entire $R$ is covered by `trumpets'. This completes the proof.
\end{proof}

We are ready to prove theorem \ref{specific form of theorem 1.3}.
\begin{proof}[Proof of theorem \ref{specific form of theorem 1.3}]
    By proposition \ref{Foliation by lower caps}, and proposition \ref{FOliation by lower trumpets} applied with both $k$ and $n -k + 1$, we can find some large constant $R_0 = R_0(n,k) > 0$ so that
    \begin{equation}
        \Sigma^{n,k}_{\sigma}, R_{n,k}(\Sigma^{n,n-k+1}_{\sigma}), \Gamma^{n,k}_{a}, R_{n,k}(\Gamma^{n,n-k+1}_a)
    \end{equation}
  foliate the region
    $$\mathbf{R}^{n+1} \setminus (B(0, R_0)\cup C_{n,k}\cup \Tilde{C}_{n,k}\cup \Lambda_{n,k})$$
    by proposition \ref{Foliation by lower caps} and proposition \ref{FOliation by lower trumpets}.
    Therefore by adding in these three extra self shrinkers into account, we obtain the desired foliation of
    $$\mathbf{R}^{n+1} \setminus B(0, R_0)$$
    thus proving theorem \ref{specific form of theorem 1.3}.
    
\end{proof}
Before ending this section, we show that our foliation also admits a neighborhood of the cylinder $C_{n,k}$ ($\Tilde{C}_{n,k}$) where the unit normals of the foliation satisfies an estimate analogous to that in lemma 4.11 of \cite{ADS}. This in particular implies that the self-shrinkers in theorem \ref{specific form of theorem 1.3} can alternatively be used to obtain the inner-outer estimates in \cite{Du}. This estimate is due to lemma \ref{Derivative bound for trumpets}, and lemma \ref{derivative estimate for caps}.
\begin{proposition}[cf lemma 4.11 of \cite{ADS}]\label{good neighborhood with good normals}
    There exists $x_0 = x_0(n,k) > 0$ so that in the following neighborhood of cylinder $\Tilde{C}_{n,k}$
    $$N_0 = \{(xw_1, yw_2)\ | \ \frac{9}{10}\beta_{n,k} \leq y < \sigma_{n,k}x, \ x \geq x_0\},$$
    if we write the unit normals as
    $$\nu(xw_1, yw_2) = (-\sin \varphi w_1, \cos \varphi w_2)$$
    then in that neighborhood
    $$\tan \varphi = \frac{x(y^2 - 2(n-k))}{y(x^2 - 2(k-1))}w$$
    with $1 \leq w \leq 1 + \frac{K}{x^2}$
    for some $K = K(n,k)$. Similar statement holds for $C_{n,k}$ as well.
\end{proposition}
\begin{proof}[Proof of proposition \ref{good neighborhood with good normals}]
The conclusion automatically follows when $y > \beta_{n,k}$ due to lemma \ref{Derivative bound for trumpets}, hence we focus on the interior of $\Tilde{C}_{n,k}$. Fix $M_0 = M_0(n,k) > 0$ from lemma \ref{derivative estimate for caps}. Then as long as $$\frac{M_0}{a} \leq \frac{\beta_{n,k}}{2},$$
we see that parts of `disks' inside $N_0$ are contained in the `intermediate region' i.e when $x \leq x_{Ma}$. Therefore, by lemma \ref{derivative estimate for caps}, we have
$$w \leq 1 + \frac{K}{x^2} + \frac{K}{a^2 - x^2},$$
where $K = K(n,k) > 0$. 
We now estimate $\frac{1}{a^2 - x^2}$. Let $p_{a} > 0$ satisfy $$u_a(p_{a}) = \frac{9}{10}\beta_{n,k}$$
where $u_a$ is from proposition \ref{Existence of caps}. By the upper bound of $u_a$ given in proposition \ref{Existence of caps}, we have
$$\frac{9}{10}\beta_{n,k} \leq \sqrt{\beta_{n,k}^2[1 - (1 - C(n,k)\frac{\ln a}{a^2})\frac{p_a^2 - c(n,k)}{a^2 - \alpha_k^2}]}.$$
Solving above inequality in $p_a$, we see that for all large enough $a$, we obtain 
$$p_a^2 \leq \frac{1}{4}a^2,$$
which immediately implies that
$$\frac{1}{a^2 - x^2}\leq \frac{1}{x^2}$$
for all $x \leq p_a$. Therefore, one can find some $a_0 = a_0(n,k) > 0$ so that as long as $a \geq a_0$, 
$$w \leq 1 + \frac{2K}{x^2}$$
for all $(xw_1, u_a(x)w_2) \in N_0$. Therefore, by letting $x_0 \geq a_0$, the condition $a \geq a_0$ is automatically satisfied for all `disks' which intersect $N_0$, thus completing the proof of proposition \ref{good neighborhood with good normals}.
\end{proof}

\section*{Acknowledgment}
The author would like to thank his advisor Prof. Nata\v sa \v Se\v sum for her constant support and guidance. He also thanks Prof. Kyeongsu Choi for his various suggestions to improve the quality of the paper.
%\section*{Disclosure of resources}
%ChatGPT was used to generate the figure 1 in the main text, and proofread the initial manuscript. Ai was not used in any other way.

\end{document}